\documentclass[13pt,a4paper, leqno]{article}
\usepackage{amssymb, amscd, amsthm, amsmath,latexsym, amstext,mathrsfs,multicol, mathtools}
\usepackage[all]{xy}
\usepackage{graphicx,enumitem}

\usepackage[colorlinks, linkcolor=blue, anchorcolor=green, citecolor=blue]{hyperref}

\newtheorem{thm}{Theorem}[section]

\newtheorem{lemma}[thm]{Lemma}
\newtheorem{prop}[thm]{Proposition}

\theoremstyle{remark}
\newtheorem{remark}[thm]{\textbf{Remark}}

\theoremstyle{definition}
\newtheorem{defn}[thm]{Definition}

\numberwithin{equation}{thm}
\newcommand{\newpara}{\noindent\refstepcounter{thm}{\bf(\thethm)\;}}  

\DeclareFontEncoding{OT2}{}{} 

\newcommand{\N}{\mathbb{N}}
\newcommand{\Q}{\mathbb{Q}}
\newcommand{\Z}{\mathbb{Z}}

\newcommand{\fm}{\mathfrak{m}}

\newcommand{\ad}{\mathrm{ad}}
\newcommand{\Alt}{\mathrm{Alt}}

\newcommand{\Br}{\mathrm{Br}}
\newcommand{\car}{\mathrm{char}}

\newcommand{\CH}{\mathrm{CH}}
\newcommand{\coker}{\mathrm{Coker}}
\newcommand{\Cor}{\mathrm{Cor}}
\newcommand{\coind}{\mathrm{coind}}

\newcommand{\disc}{\mathrm{disc}}

\newcommand{\End}{\mathrm{End}}

\newcommand{\Nrd}{\mathrm{Nrd}}
\newcommand{\Nrp}{\mathrm{Nrp}}

\newcommand{\ind}{\mathrm{ind}}

\newcommand{\hp}{\mathrm{hp}}
\newcommand{\Int}{\mathrm{Int}}

\newcommand{\Ker}{\mathrm{Ker}}

\newcommand{\rad}{\mathrm{rad}}
\newcommand{\Res}{\mathrm{Res}}
\newcommand{\rk}{\mathrm{rk}}
\newcommand{\SB}{\mathrm{SB}}
\newcommand{\Symd}{\mathrm{Symd}}
\newcommand{\Skew}{\mathrm{Skew}}
\newcommand{\ssrk}{\mathrm{ssrk}}
\newcommand{\Sym}{\mathrm{Sym}}

\newcommand{\sd}{\mathrm{sd}}
\newcommand{\Trd}{\mathrm{Trd}}

\newcommand{\rW}{\mathrm{W}}
\newcommand{\rI}{\mathrm{I}}
\newcommand{\rH}{\mathrm{H}}

\newcommand{\fp}{\mathfrak{p}}

\newcommand{\fq}{\mathfrak{q}}

\newcommand{\cD}{\mathcal{D}}

\newcommand{\sH}{\mathscr{H}}

\newcommand{\sR}{\mathscr{R}}

\newcommand{\bfGL}{\mathbf{GL}}

\newcommand{\bfSU}{\mathbf{SU}}

\newcommand{\bfPGU}{\mathbf{PGU}}

\newcommand{\et}{\text{\'et}}

\newcommand{\ilim}{\varinjlim}

\newcommand{\lra}{\longrightarrow}

\newcommand{\ov}[1]{\overline{#1}}

\newcommand{\Pic}{\mathrm{Pic}}
\newcommand{\bfSpin}{\mathbf{Spin}}
\newcommand{\bfU}{\mathbf{U}}
\newcommand{\bfO}{\mathbf{O}}

\begin{document}
\title{\textbf{Cohomological invariants of hermitian forms that detect hyperbolicity}}
\author{Yong Hu and Alexandre Lourdeaux}

\maketitle

\begin{abstract}
 By using unramified cohomology groups, we construct a full sequence of cohomological invariants for hermitian forms of any (orthogonal, symplectic or unitary) type that can be used to detect hyperbolicity. The base central simple algebra can have arbitrary degree and the base field can have arbitrary characteristic. In the orthogonal case, we work with hermitian pairs, and we apply our construction to show that over fields of separable dimension 3, hermitian pairs over quaternion algebras with trivial classical invariants are hyperbolic.  This last result extends a result of Berhuy to arbitrary characteristic.
\end{abstract}

\tableofcontents

\section{Introduction}

In the development of the algebraic theory of quadratic forms, each time the construction of a new useful cohomological invariant, such as the Clifford invariant, essentially due to Witt \cite{Witt37}, and the Arason invariant \cite{Arason75JA}, marks a milestone. Since the proof of Milnor's conjecture (\cite{Voe03MilnorConj}, \cite{OVV07}), higher degree cohomological invariants have also been well defined, so that a complete system of invariants has been available to detect the hyperbolicity of quadratic forms, over any field of characteristic different from 2. In fact, defining these invariants in characteristic 2 was made possible even earlier, thanks to Kato's work \cite{Kato82Invent}. Therefore, by Witt theory, at least theoretically one has a complete tool kit to treat the classification problem for (nondegenerate) quadratic forms over any field (at least for even-dimensional ones in characteristic 2).

The main purpose of this paper is to study the same problem for hermitian forms over central simple involutorial algebras, for which Witt theory is also applicable.

To briefly recall some previous works in this direction, let us assume for convenience that our base field has characteristic different from 2. Low degree cohomological invariants can be defined even for involutions, which correspond to similarity classes of hermitian forms. By the time that the famous book \cite{KMRT} was published, at least the following invariants have been defined and used in research: the discriminant and the Clifford algebra for orthogonal involutions, which give cohomological invariants in degrees 1 and 2 respectively; the discriminant algebra for unitary involutions, which has degree 2 as a cohomological invariant; and the Pfaffian for symplectic involutions, which serves essentially as a degree 3 invariant, as is justified in \cite{BerhuyMonsurroTignol03} and \cite{GaribaldiParimalaTignol09}. On the other hand, in their pioneering work \cite{BP2}, Bayer-Fluckiger and Parimala use the Rost invariant of related semisimple simply connected algebraic groups to define a degree 3 invariant, which they call Rost invariant, for hermitian forms in the orthogonal and unitary cases. Their idea and method, as well as those explained in the nice survey \cite{Tignol10}, are further developed to study degree 3 invariants of involutions, by Qu\'eguiner-Mathieu and Tignol \cite{QueguinerMathieuTignol15}, and by Barry, Masquelein and Qu\'eguiner-Mathieu \cite{BarryMasqueleinQueguinerMathieu22}. Among the above mentioned invariants, some are defined first in a relative setting and then in the absolute case, in a compatible way.

The classical invariants, for hermitian forms and involutions, all take values in cohomology groups up to degree 3. As far as we know, higher degree cohomological invariants have been constructed only in several special cases until the present work. An almost complete list, if the split case and the index 2 symplectic case are not counted, seems to be as follows:

\begin{itemize}
  \item invariants in all degrees for Pfister involutions (of orthogonal type in characteristic $\neq 2$) \cite{BayerParimalaQueguinerMathieu03};
  \item invariants in all degrees for quaternionic skew-hermitian forms, or equivalently orthogonal involutions, in characteristic $\neq 2$ \cite{Berhuy07ArchMath};
  \item the relative $e_4$-invariant for totally decomposable involutions on degree 8 algebras \cite{QueguinerMathieuTignol12};
  \item an absolute invariant of degree 4 for unitary involutions on degree 4 algebras (with an extra assumption on its center) \cite{RostSerreTignol2006, Tignol06}.
\end{itemize}

The main result of this paper is the construction of a full sequence of cohomological invariants $(e_n)_{n\geqslant 1}$ for hermitian forms in all the three (orthogonal, symplectic and unitary) cases in arbitrary characteristic. To be more precise, the orthogonal case refers to hermitian pairs (cf. \S\;\ref{sec2p2}), which are equivalent objects as generalized quadratic forms (or up to scalars, as quadratic pairs; see Remark\;\ref{2.9}), but which are different than orthogonal involutions in characteristic 2. Our invariants are defined absolutely, and they allow detecting the hyperbolicity of hermitian forms, at least when assuming the characteristic is not 2 in the symplectic and unitary cases. Our main theorems are stated as Theorems\;\ref{4.5}, \ref{4.15} and \ref{5.5}, showing the usefulness of our invariants.

One remark calls for special attention here. As is already known from \cite{QueguinerMathieu02Besancon} and \cite[\S\,3.4]{BayerParimalaQueguinerMathieu03}, there is no hope to define functorially an absolute $e_3$-invariant in the orthogonal case, if one insists on using a quotient of the cohomology group of the base field; see also \cite[\S\,5.3]{Tignol10} for the symplectic case. The crucial observation that makes our construction valid, which is inspired by \cite{BayerParimalaQueguinerMathieu03} and \cite{Berhuy07ArchMath}, is that the use of the unramified cohomology (cf. \S\,\ref{sec3p1}) can be a successful remedy. In fact, the unramified cohomology group we utilize can be larger than any quotient of the cohomology of the base field, but it is still canonically associated to basic input firmly tired to the problem, namely, the base field, the Brauer class of the central simple algebra over which hermitian forms are considered, and the type of hermitian forms (with orthogonal case corresponding to hermitian pairs, if one wants to be more precise). That our invariants are indeed complete for the purpose of testing hyperbolicity relies on the hyperbolicity theorems proved by Karpenko and Tignol \cite{Karpenko10DocMath, Karpenko12SciChina} in characteristic $\neq 2$, as well as a recent isotropy theorem from Medhi's thesis \cite{Medhi} for the orthogonal case in characteristic 2. This is the only place where a characteristic restriction might be needed in the symplectic and unitary cases; but see Remark\;\ref{2.17}. The construction of invariants up to degree 3 in the literature for those two cases, however, may be applied in any characteristic, in view of the results about unramified cohomology in characteristic 2 obtained recently in \cite{HuSun23}.

In the orthogonal case, we feel that a bit more work is in need if we would like to cover the case of characteristic 2. As is clear from the literature (see e.g. \cite{ElomoryTignol}, \cite{Medhi}), quadratic pairs and generalized quadratic forms are more appropriate to work with than orthogonal involutions in arbitrary characteristic. Gille \cite{Gille} introduced the concept of hermitian lifts of quadratic pairs, and discussed the Witt theory as well as  the first two classical invariants for them. We take a slightly different view on hermitian lifts and regard them as somewhat independent of lifting quadratic pairs, hence the notion of hermitian pairs. Also, we include a (standard) proof that the hyperbolicity theorem for hermitian pairs can be deduced from Medhi's isotropy theorem. As an application, our Theorem\;\ref{5.13} extends a result of Berhuy \cite[Prop.\;15]{Berhuy07ArchMath} about classification of quaternionic skew-hermitian forms over fields of cohomological dimension 3 to arbitrary characteristic.

We work with hermitian forms more often than with involutions. There are two main reasons for this. The first one is that for hermitian forms the orthogonal sum operation is naturally defined, so that our invariants induce functorial group homomorphisms. The second reason is, as is clear to experts, that defining invariants for involutions may require more restrictive assumptions, since the effect of similitudes may not be neglectable. Nevertheless, we do compare (in \eqref{4.9}, \eqref{4.19} and \S\,\ref{sec5p2}) our invariants with those defined for involutions in the literature, at least in a few cases that are important in our eyes.

In a forthcoming work we intend to apply our cohomological invariants to study Galois cohomology of classical groups over some special fields with nice arithmetic properties.

Last but not least, it is perhaps worth mentioning that, compared to a number of earlier papers, we allow non-division algebras as bases on which hermitian forms are defiend. This makes it more convenient to deal with Morita invariance and base extension functoriality for our cohomological invariants.

\

\noindent{\bf Notation and terminology}. Unless otherwise stated, we follow the notation and terminology in \cite{KMRT} for central simple algebras, involutions and algebraic groups. In particular, for a central simple algebra $A$, we denote by $\deg(A)$, $\ind(A)$ and $\exp(A)$ its degree, index and exponent respectively. The coindex of $A$ is defined as
\[
\coind(A):=\frac{\deg(A)}{\ind(A)}.
\]

The Brauer group of a field $F$ is denoted by $\Br(F)$, and for a central simple $F$-algebra $A$, we denote by $(A)\in\Br(F)$ its Brauer class.

A central simple algebra is called \emph{decomposable} if it is a tensor of two central simple algebras of smaller degree.

The opposite algebra of an algebra $D$ is denoted by $D^{op}$.

 Our main reference for  quadratic forms is \cite{EKM08}. So the group generated by Witt classes of $n$-fold quadratic Pfister forms over a field $F$ is denoted by $\rI^n_q(F)$.

Cohomology functors with values in algebraic groups, which we sometimes write with the subscript ``fppf'', refers to the cohomology with respect to the fppf topology. On the other hand, the cohomology functors with twists of $\Z/n$ or $\Q/\Z$ as coefficients are used as we explain in \S\,\ref{sec3p1}. Since the most frequently used case is associated with 2-torsion, we write $\rH^n(\cdot)$ for the functor $\rH^n\big(\cdot, \Z/2(n-1)\big)$.

For an abelian group $M$ and a positive integer $n$, we write $M[n]$ for the subgroup of $n$-torsion elements in $M$. For example, $\Br(F)[2]$ denotes the 2-torsion subgroup of the Brauer group $\Br(F)$.

If $K/F$ is a field extension, the restriction map between cohomology groups of $F$ and $K$ is usually denoted by $\Res_{K/F}$. A cohomology class over $K$ is said to \emph{descend} (\emph{uniquely}) to $F$ if it is the image under the restriction map $\Res_{K/F}$ of a (\emph{unique}) cohomology class over $F$.

Other notation, if non-standard, will be introduced along with later discussions. For example, the unramified cohomology groups $\rH^n_+(D), \rH^n_-(D)$  associated to a central simple algebra $D$, and $\rH^n(D'/F)$ associated to a central simple algebra $D'$ over a quadratic separable extension of a field $F$, will be defined in \eqref{3.6}.

\section{Involutions, hermitian forms and hyperbolicity}

This section is devoted to preliminary reviews on basic definitions and useful facts about involutorial algebras, hermitian forms and other related objects. Most of the content in this section is known in the literature. The only exceptions, if there are any, are perhaps the discussion about the Morita theory for hermitian pairs in \S\;\ref{sec2p2} and the characteristic free version of the hyperbolicity theorem for hermitian pairs (Theorem\;\ref{2.16}).

\medskip

Throughout this section, let $F$ be a field (of any characteristic).

\subsection{Involutorial algebras and hermitian forms}

\newpara\label{2.1} Let $A$ be a finite dimensional simple $F$-algebra. By the Wedderburn--Artin theorem,  there is a unique simple (right) $A$-module $P$ up to isomorphism. Thus, for any finitely generated (right) $A$-module $V$, we define the \emph{semisimple rank} $\ssrk_A(V)$ of $V$ to be the unique natural number $r\in\N$ such that $V\cong P^{\oplus r}$. Letting $K=Z(A)$ be the center of $A$, we have
\[
\ssrk_A(V)=\frac{\dim_K(V)}{\deg(A)\ind(A)}\,.
\]Here  $\deg(A)$ and $\ind(A)$ denote the degree and the (Schur) index of $A$ as a central simple $K$-algebra. For example, if $A=\mathrm{M}_n(F)$ is a matrix algebra, then $\ssrk_A(A)=n$.

Clearly, if $A=D$ is a division algebra, then the semisimple rank of any finitely generated $D$-module $V$ is the same as the usual dimension $\dim_D(V)$ of $V$ as a $D$-vector space.

Now assume more generally that $A$ is a finite dimensional semisimple $F$-algebra. Then $A$ is isomorphic to a finite product $A_1\times\cdots\times A_s$, where each $A_i$ is a finite dimensional simple $F$-algebra whose center $K_i:=Z(A_i)$ is a finite field extension of $F$. Any finitely generated $A$-module $V$ can be written as $V=V_1\times\cdots\times V_s$, where each $V_i$ is a finitely generated $A_i$-module. Then we define the semisimple rank $\ssrk_A(V)$  as the $s$-tuple
\[
\ssrk_A(V):=\big(\ssrk_{A_1}(V_1),\cdots, \ssrk_{A_s}(V_s)\big)\,\in\N^{\oplus s}\,.
\]We say that $V$ has (constant) \emph{semisimple rank} $r$ if $\ssrk_{A_i}(V_i)=r$  for every $1\leqslant i\leqslant s$. In this case we write $\ssrk_A(V)=r$.

Note that if $A$ is a central simple $F$-algebra and $F'/F$ is a field extension, then for any finitely generated $A$-module $V$, we have
\[
\ind(A)\ssrk_A(V)=\deg\big(\End_A(V)\big)=\ind(A_{F'})\ssrk_{A'}(V_{F'}),
\]where $A_{F'}=A\otimes_FF'$ and $V_{F'}=V\otimes_FF'$.  \hfill $\blacksquare$

\

\newpara\label{2.2} By an \emph{involutorial $F$-algebra} we mean a pair $(A,\,\sigma)$ consisting of a nonzero finite dimensional $F$-algebra $A$ and an involution $\sigma$ on $A$ whose restriction to $F$ is the identity. For an involutorial algebra $(A,\sigma)$ we put
\[
\begin{split}
\Sym(A,\sigma)&=\{a\in A\,|\,\sigma(a)=a\}\,,\\
\Symd(A,\sigma)&=\{a+\sigma(a)\,|\,a\in A\}\,,\\
\Skew(A,\sigma)&=\{a\in A\,|\,\sigma(a)=-a\}\,,\\
\text{and }\;\;\Alt(A,\sigma)&=\{a-\sigma(a)\,|\,a\in A\}.
\end{split}
\]

By a \emph{central simple involutorial algebra} (CSIA) over $F$, we mean an involutorial $F$-algebra $(A,\,\sigma)$ satisfying the following conditions:

\begin{itemize}[nosep]
\item $A$ is separable as an $F$-algebra.
  \item An element in the center $Z(A)$ is fixed by $\sigma$ if and only if it lies in $F$.
  \item There is no nonzero two-sided ideal $I$ of $A$ such that $I\neq A$ and $\sigma(I)\subseteq I$.
\end{itemize}

Let $A$ be a nonzero finite dimensional $F$-algebra with center $K=Z(A)$ and let $\sigma$ be an involution on $A$. Then the pair $(A,\,\sigma)$ is a CSIA over $F$ if and only if one of the following two cases happens:

\begin{enumerate}
  \item There is a central simple $F$-algebra $B$, with opposite algebra $B^{op}$, and an $F$-algebra isomorphism $f: B\times B^{op}\overset{\sim}{\longrightarrow} A$ such that $f^{-1}\circ \sigma \circ f$ is the exchange involution $(x,\,y)\mapsto (y,\,x)$ on $B\times B^{op}$.

  \item $K$  is a field, $A$ is a central simple $K$-algebra and $F=K^{\sigma}:=\{a\in K\,|\,\sigma(a)=a\}$.

\end{enumerate}

The involution $\sigma$ is called of the \emph{first kind} (resp. \emph{second kind}) if $\dim_FK=1$ (resp. $\dim_FK=2$). For involutions of the first kind, we distinguish the \emph{orthogonal type} and the \emph{symplectic type} as in \cite[(2.5)]{KMRT}. If $\sigma$ is of the second kind, we also say that it is \emph{unitary}. By the \emph{type} of a central simple involutorial algebra $(A,\sigma)$ we mean the type of the involution $\sigma$.

When $K$ is not a field, following the terminology of \cite{BecherGrenierBoleyTignol18JA}, we say that $\sigma$ is \emph{unitary of inner type}. If $K$ is a field, then the field extension $K/F$ is separable of degree $\leqslant2$, and $\sigma$ is unitary when $K/F$ is a quadratic extension.

To emphasize the relation between $\sigma$ and $K/F$, we often say that $\sigma$ is a $K/F$-\emph{involution} (of whichever type).\hfill $\blacksquare$

\

\newpara\label{2.3} Let $(A,\,\sigma)$ be a CSIA over $F$. A \emph{hermitian form} over $(A,\,\sigma)$ is a pair $(V,\,h)$, where $V$ is a finitely generated right $A$-module of constant semisimple rank and $h: V\times V\to A$ is a function such that the following properties hold for all $x,\,y,\,z\in V$ and all $\lambda \in A$:
\[
h(x+y,\,z)=h(x,\,z)+h(y,\,z)\;,\; h(x\lambda,\,y)=\sigma(\lambda)h(x,\,y)\;,\;\; h(y,\,x)=\sigma(h(x,\,y))\,.
\]We also say that $h$ is a hermitian form on the $A$-module $V$ \emph{with respect to the involution} $\sigma$, or that $h$ is a $\sigma$-\emph{hermitian form} on $V$. For simplicity, sometimes a hermitian form is simply denoted by $h$.

The (\emph{semisimple}) \emph{rank} $\rk(h)$ of a hermitian form $(V,h)$ is defined as the semisimple rank of the $A$-module $V$, i.e.,
\[
\rk(h):=\ssrk_A(V).
\]

The \emph{radical} of a hermitian form $(V,h)$ is defined to be the set
\[
\rad(h):=\{x\in V\,|\, h(x,\,y)=0 \text{ for all } y\in V\}\,.
\]We say that $h$ is \emph{nondegenerate} if $\rad(h)=0$. If there is a nonzero $x\in V$ such that $h(x,\,x)=0$, then we say that $h$ is \emph{isotropic}; otherwise we say that $h$ is \emph{anisotropic}. Clearly, any anisotropic hermitian form is nondegenerate.

For an $A$-submodule $U\subseteq V$, we define its \emph{orthogonal complement} $U^\bot$ in $V$ (with respect to $h$) by
\[
U^\bot:=\{x\in V\,|\,h(x,\,y)=0 \text{ for all } y\in U\}\,.
\]If $U=U^\bot$, then we say that $U$ is a \emph{lagrangian} of $(V,\,h)$. If $(V,h)$ is nondegenerate and has a lagrangian, then we say that it is  \emph{metabolic}. Clearly, if $h$ is metabolic with $V\neq 0$, then $h$ is isotropic.

If $\sigma$ is unitary of inner type, one can show that nondegenerate hermitian forms over $(A,\,\sigma)$ are all metabolic. So, when studying anisotropic forms we are mostly interested in the case where $K=Z(A)$ is a field and $A$ is a central simple $K$-algebra.

A hermitian form $(V,h)$ over $(A,\sigma)$ is called \emph{even} if $h(x,x)\in \Symd(A,\sigma)=\{a+\sigma(a)\,|\,a\in A\}$ for all $x\in V$ (cf. \cite[Chap.\;I, (3.1)]{Knus}). If the base field $F$ has characteristic $\car(F)\neq 2$ or if $\sigma$ is unitary, then all hermitian forms over $(A,\sigma)$ are even.

An even metabolic hermitian form is called \emph{hyperbolic}.

The isomorphism classes of nondegenerate hermitian forms over $(A,\sigma)$ form a commutative monoid with respect to the orthogonal sum operation. The quotient of the Grothendieck group of this monoid by the subgroup generated by the classes of hyperbolic forms is called the \emph{Witt group of hermitian forms} over $(A,\sigma)$ (cf. \cite[\S\,I.10]{Knus}). We denote this group by $\rW(A,\sigma)$.

The class of any metabolic form is trivial in the Witt group, by \cite[Chap.\;I, (3.7.6)]{Knus}.

The \emph{Witt group of even hermitian forms} over $(A,\sigma)$ can be defined in a similar way. We denote this group by $\rW_+(A,\sigma)$.

Note that $\rW(A,\sigma)=\rW_+(A,\sigma)=0$ if $\sigma$ is unitary of inner type. \hfill $\blacksquare$

\

\newpara\label{2.4} Let $(B,\,\gamma)$ be a CSIA over $F$. Let $(W,f)$ be a nondegenerate hermitian form over $(B,\gamma)$ and put $A=\End_B(W)$. The \emph{adjoint involution} associated to $f$ (and $\gamma$) is the involution $\ad_f=\ad_{f/\gamma}$ on $A$ determined by the following condition:
\[
\forall\;x,y\in W,\;\forall\;\alpha\in A=\End_B(W),\quad f\big(x,\alpha(y)\big)=f\big(\ad_f(\alpha)(x),y\big).
\]
For any nondegenerate hermitian form $(V,h)$ over $(A,\ad_f)$, the map
\[
f\star h: (V\otimes_AW)\times (V\otimes_AW)\longrightarrow B
\]given by
\[
f\star h (v_1\otimes w_1,\,v_2\otimes w_2)=f\big(w_1,h(v_1,v_2)(w_2)\big)
\]is a nondegenerate hermitian form over $(B,\gamma)$.

For any fixed $(W,f)$, one can check that the \emph{Morita functor}
\[
(V,h)\longmapsto (V^{\fm},\,h^{\fm}):=(V\otimes_AW,\,f\star h)
\]preserves isotropic forms, metabolic forms, even forms and so on. This functor is in fact an equivalence of categories and its quasi-inverse is also a Morita functor (\cite[\S\;I.9]{Knus}). In particular, we have  group isomorphisms
\[
\rW(A,\ad_{f/\gamma})\overset{\sim}{\longrightarrow}\rW(B,\gamma)\quad\text{and}\quad \rW_+(A,\ad_{f/\gamma})\overset{\sim}{\longrightarrow}\rW_+(B,\gamma),
\]
induced by the above Morita equivalence. \hfill $\blacksquare$

\

\newpara\label{2.5} Let $(B,\,\gamma)$ be a CSIA over $F$.   Assume that $K=Z(B)$ is a field (i.e., $\gamma$ is not unitary of inner type). Let $W$ be a finitely generated right $B$-module and put $A=\End_B(W)$. Let $\sigma$ be a $K/F$-involution on $A$ which is of the same (orthogonal, symplectic or unitary) type as $\gamma$. By \cite[(4.2)]{KMRT}, there exists a nondegenerate hermitian form $f$ on $W$ with respect to $\gamma$ such that $\sigma=\ad_{f/\gamma}$, and up to a scalar multiple from $F^*$, the form $f$ is uniquely determined by $\sigma$ and $\gamma$. \hfill $\blacksquare$

\subsection{Quadratic and hermitian pairs}\label{sec2p2}

\newpara\label{2.6} Recall the following definition from {\cite[(5.4)]{KMRT}}: Given a central simple $F$-algebra $A$, a \emph{quadratic pair} on $A$ is a pair $(\sigma,f)$, where $\sigma$ is an involution of the first kind  on $A$ and $f$ is an $F$-linear map
\[ f : \mathrm{Sym}(A,\sigma) \to F \]
 such that

\begin{itemize}[nosep]
  \item $\dim_F\Sym(A,\sigma)=\frac{n(n+1)}{2}$, where $n=\deg(A)$;
  \item $\Trd_A(\Skew(A,\sigma))=0$, where $\Trd_A:A\to F$ denotes the reduced trace map;
  \item $f(x+\sigma(x)) = \mathrm{Trd}_A(x)$ for all $x \in A$.
\end{itemize}
 The map $f$ is called a \emph{semi-trace} for the pair $(A,\sigma)$, and the triple $(A,\sigma,f)$ is called a \emph{central simple algebra with quadratic pair} (CSAQP for short) over $F$.

If $\car(F)\neq 2$, then the involution $\sigma$ must be orthogonal and the semi-trace $f$ is uniquely determined by $f=\frac{1}{2}\Trd_A$.

If $\car(F)=2$, then the involution must be symplectic and hence $\deg(A)$ must be even.

The isotropy and hyperbolicity properties of a CSAQP are defined in  \cite[(6.5) and (6.12)]{KMRT}.

Although not uniquely determined by the summands, orthogonal sums of  two given  central simple algebras with quadratic pair can be defined as in \cite[Def.\;1.4]{BerhuyFringsTignol}.

Let us define the \emph{degree} of a central simple algebras with quadratic pair $(A,\sigma,f)$ to be the degree of the underlying algebra $A$. If $(A,\sigma,f)$ is an orthogonal sum of $(A_1,\sigma_1,f_1)$ and $(A_2,\sigma_2,f_2)$, then  $A_1, A_2$ and $A$ are all Brauer equivalent, and $\deg(A)=\deg(A_1)+\deg(A_2)$ (\cite[Prop.\;1.2]{BerhuyFringsTignol}). \hfill $\blacksquare$

\begin{prop}[{\cite[Prop. 1.11]{BerhuyFringsTignol}}]\label{2.7}
Every CSAQP is a direct sum of an anisotropic one and an hyperbolic one.

Also, if $\mathfrak{q}$ is an orthogonal sum of two CSAQPs $\mathfrak{q}_1$ and $\mathfrak{q}_2$, and if $\mathfrak{q}$ and $\mathfrak{q}_1$ are hyperbolic, then  $\mathfrak{q}_2$ is hyperbolic.
\end{prop}

\begin{defn}\label{2.8} Let $(D,\theta)$ be a CSIA of the first kind over $F$. Assume that it is  orthogonal if  $\car(F)\neq 2$.

For a hermitian form $(V,h)$ over $(D,\theta)$, we  write simply $\mathrm{Sym}(h)=\mathrm{Sym}(\mathrm{End}_D(V),\mathrm{ad}_{h/\theta})$.

A \emph{hermitian pair} over $(D,\theta)$ is a pair $(h,f)$, where $h=(V,h)$ is a nondegenerate hermitian form over $(D,\theta)$, and $f$ is a semi-trace for $(\End_D(V),\mathrm{ad}_{h/\theta})$.

Note that in order for $(h,f)$ to be a hermitian pair, $h$ must be even.

A hermitian pair $(h,f)$ over $(D,\theta)$ is said to be \emph{isotropic} if  $(\mathrm{End}_E(V),\mathrm{ad}_{h/\theta},f)$ is isotropic as a CSAQP. It is called \emph{anisotropic} otherwise.

Similarly, $(h,f)$ is called \emph{hyperbolic} if $(\mathrm{End}_D(V),\mathrm{ad}_{h/\theta},f)$ is hyperbolic.

The (\emph{semisimple}) \emph{rank} of a hermitian pair $(h,f)$ is defined as the rank of the hermitian form $h$.

Let $(A,\sigma,f)$ be a CSAQP. If $h=(V,h)$ is a nondegenerate hermitian form over $(D,\theta)$ such that $(A,\sigma)=(\End_D(V),\mathrm{ad}_{h/\theta})$, then, as in \cite[\S\,7.4.2]{Gille}, we say that the hermitian pair $(h,f)$ is a \emph{hermitian lift} of $(A,\sigma,f)$. \hfill $\blacksquare$
\end{defn}

\begin{remark}\label{2.9}
  Another notion that is often used to lift quadratic pairs is that of \emph{generalized quadratic form}; see e.g. \cite{ElomoryTignol}. As explained in \cite[\S\,1.3]{Medhi}, generalized quadratic forms can be viewed as equivalent to hermitian pairs. \hfill $\blacksquare$
\end{remark}

\newpara\label{2.10} Let $(D,\theta)$ be a CSIA of the first kind over $F$, of  orthogonal type if  $\car(F)\neq 2$.

Let $(h_1,f_1)$ and $(h_2,f_2)$ be hermitian pairs over $(D,\theta)$. Their \emph{orthogonal sum}, denoted by $(h_1,f_1)\perp (h_2,f_2)$, is the hermitian pair $(h,f)$ defined as follows: Let $V_i$ be the underlying $D$-module on which $h_i$ is defined for $i=1,2$. The hermitian form $h$ is taken to be the orthogonal sum $h_1\perp h_2$, defined on $V:=V_1\oplus V_2$. Thanks to \cite[Prop. 7.4.2]{Gille}, there exists a unique semi-trace $f$ on $(\mathrm{End}_D(V),\mathrm{ad}_h)$ such that $f : \mathrm{Sym}(h) \to F$ restricts to $f_1$ and $f_2$ via the natural maps \[ \mathrm{Sym}(h_i) \hookrightarrow \mathrm{Sym}(h), \; i=1,2 . \] Note that \cite[Prop. 7.4.2]{Gille} is written for $D$ a division algebra, but the proof works exactly the same for a general $D$. Thus we may define the orthogonal sum $(h_1,f_1) \perp (h_2,f_2)$ as $(h,f)$.

Actually, by \cite[(5.7)]{KMRT}, for $i=1,2$ we can choose an element  $l_i \in \mathrm{End}_D(V_i)$  such that
\[
f_i(s)=\mathrm{Trd}_{\mathrm{End}_D(V_i)}(l_i s)\quad\text{ for all } s \in \mathrm{Sym}(h_i).
\] Then the semi-trace of the orthogonal sum of $(h_1,f_1)$ and $(h_2,f_2)$ is given by
\[
f(s)=\mathrm{Trd}_{\mathrm{End}_D(V)}(ls)\quad\text{ for } l=l_1 \oplus l_2.
\]

The CSAQP associated to $(h,f)=(h_1,f_1)\perp (h_2,f_2)$ is an orthogonal sum of the two  associated to $(h_i,f_i)$. The advantage of forming the orthogonal sum in terms of hermitian pairs is that this orthogonal sum is uniquely determined by its summands.

Having defined the orthogonal sum operation for hermitian pairs, we can define the \emph{Witt group of hermitian pairs} over $(D,\theta)$ in the natural way. We denote this group by $\rW_{\hp}(D,\theta)$. \hfill $\blacksquare$

\medskip

\newpara\label{2.11} Let $(D,\theta)$ be a CSIA over $F$ as in \eqref{2.10}. Let $(V_0,h_0)$ be a nondegenerate even hermitian form over $(D,\theta)$, and put $(E,\tau)=(\End_D(V_0),\,\ad_{h_0/\theta})$.

Let $((W,h),f)$ be a hermitian pair over $(E,\tau)$. Morita equivalence for hermitian forms (cf. \eqref{2.4}) implies that the map
\[
\phi=\phi_{h,f} : \mathrm{End}_E(W)  \longrightarrow \mathrm{End}_D(W\otimes_E V_0), \quad\alpha \longmapsto  \alpha\otimes \mathrm{id}_{V_0}
\] induces an isomorphism of involutorial algebras
\[
\big(\mathrm{End}_E(W),\mathrm{ad}_{h/\tau}\big) \overset{\sim}{\to} \big(\mathrm{End}_D(W \otimes_E V_0), \mathrm{ad}_{h_0 \star h/\theta}\big).
\] The map $\phi$ induces an $F$-linear isomorphism between the spaces of symmetric elements, whence the possibility to define
\[ f' = f \circ \phi^{-1} : \mathrm{Sym}(h_0 \star h) \longrightarrow F . \]
It is easy to check that $f'$ is a semitrace for $(\mathrm{End}_D(W \otimes_E V_0),\mathrm{ad}_{h_0 \star h})$, so we obtain a Morita functor
\[
((W,h),f)\longmapsto ((W^{\fm},h^{\fm}),f^{\fm}):=\big((W\otimes_E V_0,h_0 \star h),f'\big)
\]
from the category of hermitian pairs over $(E,\tau)$ to that over $(D,\theta)$. Similar to the case of hermitian forms, the above Morita functor preserves orthogonal sums and the isotropy and hyperbolicity properties; it is an equivalence of categories and induces an isomorphism of Witt groups
\[
\rW_{\hp}(E,\tau)\overset{\sim}{\longrightarrow} \rW_{\hp}(D,\theta).
\]Note that the above Morita functor depends on the choice of the hermitian form $h_0$. \hfill $\blacksquare$

\subsection{Generic index reduction and hyperbolicity theorems}

\newpara\label{2.12} Let $m$ be a positive integer, $\alpha\in\Br(F)$ and $K/F$ a field extension. Following \cite{Blankchet91}, we say that $K$ is a $\frac{1}{m}$-\emph{splitting field}  of $\alpha$ if $\ind(\alpha_K)\,|\,m$. If moreover the free composite $KE$ is a purely transcendental extension of $E$ for every $\frac{1}{m}$-splitting field $E$, then we say that $K$ is a \emph{generic $\frac{1}{m}$-splitting field} of $\alpha$.

  A \emph{generic splitting field} is a generic $\frac{1}{m}$-splitting field for $m=1$.

Let $A$ be a central simple $F$-algebra that represents the Brauer class $\alpha$. Let  $X=\mathrm{SB}_m(A)$ be the generalized Severi--Brauer $F$-variety of right ideals of reduced dimension $m$ in $A$ (cf. \cite[(1.16)]{KMRT}). Then, by \cite[Theorem\;2]{Blankchet91}, the function field $F(X)$ is a generic $\frac{1}{m}$-splitting field of $\alpha$. (For $m=1$, this was first proved by Amitsur \cite[Thm. A]{Amitsur}.)\hfill $\blacksquare$

\medskip

\newpara\label{2.13} Let $F'/F$ be a quadratic separable field extension and let $\alpha'\in\Br(F')$.  By an $F$-\emph{based splitting field} of $\alpha'$ we mean a field extension $K/F$ such that $K\otimes_FF'$ is a splitting field of $\alpha'$. A \emph{generic $F$-based splitting field} is an $F$-based splitting field $K$ such that for every $F$-based splitting field $E$, the free composite $KE$ is a purely transcendental extension of $E$.

Let $A'$ be a central simple $F'$-algebra that represents the Brauer class $\alpha'$. Let  $X=\mathrm{R}_{F'/F}\mathrm{SB}(A')$ be the Weil restriction to $F$ of the Severi--Brauer $F'$-variety of $A'$. Then the function field $F(X)$ is a generic $F$-based splitting field of $\alpha'$.

Indeed, if $E$ is another $F$-based splitting field of $\alpha'$, then $E':=E\otimes_FF'$ is a splitting field of $A'_E:=A'\otimes_FE=A'\otimes_{F'}E'$, and the free composite $E\cdot F(X)$ is the function field of the $E$-variety $X_E:=X\times_FE=\mathrm{R}_{E'/E}\SB(A'_E)$. Since $A'_E$ is split over $E'$, the Severi--Brauer variety $\SB(A'_E)$ is a projective space over $E'$, and hence $X_E=\mathrm{R}_{E'/E}\SB(A'_E)$ is rational over $E$. \hfill $\blacksquare$

\begin{defn}[{\cite{BlackQueguinerMathieu14}}]\label{2.14}
A  generic splitting (resp. $\frac{1}{2}$-splitting) field of a Brauer class or a central simple algebra is also called a \emph{generic index reduction field of orthogonal} (resp. \emph{symplectic}) \emph{type} of it.

For a central simple $F$-algebra $A$, we define
\[
  F_+(A):=F(\SB(A))\quad \text{ and }\quad F_-(A):=F(\SB_2(A)).
\]Namely, $F_+(A)$ is the function field of the Severi--Brauer variety $\SB(A)$, and $F_-(A)$ is the function field of the generalized Severi--Brauer variety $\SB_2(A)$. As we have said in \eqref{2.12}, $F_+(A)$ (resp. $F_-(A)$) is a generic index reduction field of orthogonal (resp. symplectic) type of the Brauer class $(A)\in\Br(F)$.

Let $F'/F$ be a quadratic separable extension and $\alpha'\in \Br(F')$. Then a generic $F$-based splitting field of $\alpha'$ is also called a \emph{generic index reduction field of $F$-unitary type} of $\alpha'$.

For a central simple $F'$-algebra $A'$, we denote by
\[
  F(A'):=F\big(\mathrm{R}_{F'/F}\SB(A')\big)
\]the function field of the $F$-variety $\mathrm{R}_{F'/F}\SB(A')$. We have seen in \eqref{2.13} that $F(A')$ is a generic index reduction field of $F$-unitary type of the Brauer class $(A')\in\Br(F')$.

Let $(D,\theta)$ be a CSIA over $F$ such that $Z(D)$ is a field. A \emph{generic index reduction field} of $(D,\theta)$ is a generic index reduction field of the same type as $\theta$ of the Brauer class $(D)$.  \hfill $\blacksquare$
\end{defn}

Later in the paper we will need the following theorem due to Karpenko and Tignol, which we state in terms of hermitian forms instead of involutions.

\begin{thm}[{\cite[Thms.\;1.1 and A.1]{Karpenko10DocMath}, \cite[Thm.\;1.1]{Karpenko12SciChina}}]\label{2.15}
  Let $(D,\theta)$ be a CSIA over $F$ such that $Z(D)$ is a field. Assume $\car(F)\neq 2$.

  Then a nondegenerate hermitian form $h$ over $(D,\theta)$ is hyperbolic if and only if it is hyperbolic after base extension to a generic index reduction field of $(D,\theta)$.
\end{thm}

The analog of Theorem\;\ref{2.15} for the isotropy property is also true, according to \cite{Karpenko10ComptesRendu}, \cite{Karpenko13AJM} and \cite{KarpenkoZhykhovich13}. That isotropy theorem is not  directly used in this paper. However, we will use a recent isotropy theorem of Medhi to establish the following orthogonal case of Theorem\;\ref{2.15} in characteristic 2.

\begin{thm}\label{2.16}
Assume $\car(F)=2$. Let $(D,\theta)$ be a CSIA  of symplectic type over $F$.

Then  a hermitian pair over $(D,\theta)$ is hyperbolic if and only if it is hyperbolic after base extension to a generic index reduction field of orthogonal type of $D$.
\end{thm}
\begin{proof}
In order that evidence could be cited directly from the literature, let us prove the equivalent form of the theorem that is stated in terms of quadratic pairs.

Let $\mathfrak{q}:=(A,\sigma,f)$ be a CSAQP over $F$, and let  $K$ be a generic splitting field of $A$. Suppose that the base extension $\mathfrak{q}_K$ of $\mathfrak{q}$ is hyperbolic.  We need to prove that $\mathfrak{q}$ itself is already hyperbolic.

By Prop.\;\ref{2.7}, $\mathfrak{q}$ can be written as an orthogonal sum of two  central simple algebras with quadratic pair $\mathfrak{q}_1,\mathfrak{q}_2$ with $\mathfrak{q}_1$ hyperbolic  and $\mathfrak{q}_2$ anisotropic. We also see from Prop.\;\ref{2.7} that $(\mathfrak{q}_2)_K$ is hyperbolic. According to \cite[Prop.\;1.2]{BerhuyFringsTignol}, the underlying central simple algebras $A_1$ and $A_2$ of $\mathfrak{q}_1$ and $\mathfrak{q}_2$ are Brauer equivalent to $A$. So $K$ is also a generic splitting field of $A_i$ for $i=1,2$. Therefore, replacing $\mathfrak{q}$ with $\mathfrak{q}_2$ if necessary, we may assume that $\mathfrak{q}$  is anisotropic.

By the isotropy theorem for quadratic pairs, recently proved in \cite[Th. 4.3.1]{Medhi},  $\fq_{F'}$ is isotropic for some odd degree field extension $F'/F$. Write $\fq_{F'}$ as an orthogonal sum of two central simple $F'$-algebras with quadratic pair, say $\fq_1'$ and $\fq'_2$, with  $\fq_1'$ hyperbolic and $\fq'_2$ anisotropic. Then $\deg(\fq'_2)<\deg(\fq_{F'})$ because $\fq_{F'}$ is isotropic. Moreover, the composite field $K':=KF'$ is a generic splitting field of $A_{F'}$ such that $\fq_{K'}$ is hyperbolic, and hence $(\fq'_2)_{K'}$ is hyperbolic. By induction on the degree, we can deduce that $\fq'_2$ is hyperbolic, which implies that $\fq_{F'}=\fq'_1$ is hyperbolic.
But then, by the Springer theorem for hyperbolicity (\cite[Th. 1.14]{BerhuyFringsTignol}), $\fq$ itself must be hyperbolic. This concludes the proof.
\end{proof}

The special case of Theorem \ref{2.16} with $D$ a quaternion algebra was first obtained in \cite[Th. 8.2]{BecherDolphin}.

\begin{remark}\label{2.17}
  In the symplectic and unitary cases, it is very likely that Theorem\;\ref{2.15} remains valid in characteristic 2. As Medhi and Qu\'equiner-Mathieu have recently communicated to us, they believe that the symplectic case can be treated by modifying Tignol's argument in the appendix of \cite{Karpenko10DocMath}, and that the unitary case can be proved using techniques similar to those in Medhi's thesis \cite{Medhi}, thanks to the availability of motivic Steenrod operations in characteristic 2. \hfill $\blacksquare$
\end{remark}

\section{Unramified cohomology and invariants of quadratic forms}

In this section, let $F$ be a field.

\subsection{Unramified cohomology for generic index reduction fields}\label{sec3p1}

Let us briefly recall some basic definitions and facts about unramified cohomology in arbitrary characteristic. For more details of general discussions on unramified cohomology, we refer the reader to \cite{CT95} and \cite{ColliotHooblerKahn97}. Some useful facts about unramified cohomology in positive characteristic have been collected in \cite[\S\,3]{HuSun23}.

\

\newpara\label{3.1} Let $p$ be a prime number. Given integers $r\geqslant 1$ and $i\geqslant 0$, we define
\[
\Z/p^r(i):=\begin{cases}
  \mu_{p^r}^{\otimes i} \quad & \text{ if } \car(F)\neq p,\\
  \nu_r(i)[-i] \quad & \text{ if } \car(F)=p\,,
\end{cases}
\]as an object in the derived category of \'etale sheaves over $F$. Here, if $\car(F)\neq p$, $\mu_{p^r}$ denotes the \'etale sheaf of $p^r$-th roots of unity over $F$; and if $\car(F)=p$,
 let $\nu_r(i)=W_r\Omega^i_{\log}$ denote the $i$-th logarithmic Hodge--Witt sheaf on the big \'etale site of $F$ (\cite{Illusie79}, \cite{Shiho07}).

For every  $n\in\N$, we have the cohomology functor $H^n\bigl(\cdot\,,\,\Z/p^r(i)\bigr):=H^n_{\et}\bigl(\cdot\,,\,\Z/p^r(i)\bigr)$ on $F$-schemes.
The Zariski sheaf associated to the presheaf $U\mapsto H^n\bigl(U,\,\Z/p^r(i)\bigr)$ is denoted by $\sH^n_{p^r}(i)$.

For a smooth connected $F$-variety $X$, we define the \emph{unramified cohomology group}
\[
\rH^n_{nr}\bigl(X,\,\Z/p^r(i)\bigr):=\rH^0_{\textrm{Zar}}\bigl(X,\,\sH^n_{p^r}(i)\bigr)\,.
\]This group can also be described by using the Cousin complex of $X$. It is naturally a subgroup of $\rH^n\bigl(F(X),\,\Z/p^r(i)\bigr)$, where $F(X)$ denotes the function field of $X$.

We will be mostly interested in the case $n=i+1$, and in this case, we can give another description of the unramified cohomology group
$\rH^{i+1}_{nr}\bigl(X,\,\Z/p^r(i)\bigr)$ as follows:

Let $K$ be a field extension of $F$ and let $v$ be a discrete valuation on $K$ that is trivial on $F$. Let $\mathcal{O}_v$ denote the valuation ring of $v$ in $K$. The natural map \[
\rH^{i+1}\big(\mathcal{O}_v,\Z/p^r(i)\big)\longrightarrow \rH^{i+1}\big(K,\Z/p^r(i)\big)
\]is injective, so that we may view $\rH^{i+1}\big(\mathcal{O}_v,\Z/p^r(i)\big)$ as a subgroup of $\rH^{i+1}\big(K,\Z/p^r(i)\big)$. An element of $\rH^{i+1}(K,\Z/p^r(i))$ is called \emph{unramified} at $v$ if it lies in the subgroup $\rH^{i+1}\big(\mathcal{O}_v,\Z/p^r(i)\big)$. If $\widehat{\mathcal{O}}_v$ denotes the completion of $\mathcal{O}_v$ and $K_v$ denotes the fraction field of $\widehat{\mathcal{O}}_v$, then an element of $\rH^{i+1}(K,\Z/p^r(i))$ is unramified at $v$ if and only if its image in  $\rH^{i+1}(K_v,\Z/p^r(i))$ lies in the subgroup $\rH^{i+1}\big(\widehat{\mathcal{O}}_v,\Z/p^r(i)\big)$.

Now suppose that $K=F(X)$ is the function field of a smooth connected $F$-variety $X$. Then each codimension 1 point $x$ in $X$ determines a discrete valuation $v_x$ on the function field $K$, whose valuation ring is the local ring $\mathscr{O}_{x}$ of $X$ at $x$. The unramified cohomology group $\rH^{i+1}_{nr}\big(X,\Z/p^r(i)\big)$ can be identified with the intersection
\[
\bigcap_{x\in X^{(1)}} \rH^{i+1}\big(\mathscr{O}_x,\Z/p^r(i)\big)
\]inside the group $ \rH^{i+1}\big(K,\Z/p^r(i)\big)$, where $x$ runs through the set $X^{(1)}$ of codimension 1 points of $X$. In other words, the unramified cohomology group $ \rH^{i+1}_{nr}\big(X,\Z/p^r(i)\big)$ is the subgroup of $\rH^{i+1}\big(F(X),\Z/p^r(i)\big)$ consisting of elements that are unramified at all codimension 1 points of $X$.

The unramified cohomology groups are stable birational invariants for smooth proper varieties \cite[Thms.\;8.5.1 and 8.6.1]{ColliotHooblerKahn97}. That is, if $X$ and $X'$ are smooth proper connected $F$-varieties such that  for some $m,n\in\N$, the products $X\times_F\mathbb{P}_F^n$ and $X'\times_F\mathbb{P}_F^m$ (with the projective spaces of dimensions $m$ and $n$) are birationally equivalent over $F$,  then there is a natural isomorphism
\[
\rH^{i+1}_{nr}\big(X,\Z/p^r(i)\big)\cong \rH^{i+1}_{nr}\big(X',\Z/p^r(i)\big).
\]

Let us call a field extension $K/F$ \emph{nice} if $K$ is isomorphic to the function field $F(X)$ of some smooth proper connected $F$-variety $X$. When it is the case, the variety $X$ is called a \emph{nice model} of $K/F$. By the birational invariance mentioned above, we can define the \emph{unramified cohomology group} for any nice extension $K/F$ by
\[
\rH^{i+1}_{nr}\big(K/F,\Z/p^r(i)\big):=\rH^{i+1}_{nr}\big(X,\Z/p^r(i)\big),
\]where $X$ is any nice model. \hfill $\blacksquare$

\

\newpara\label{3.2} Fix $i\in\N$. We define $\Q_p/\Z_p(i):=\ilim_r\Z/p^r(i)$ for every prime number $p$ and
\[
\Z/m(i):=\bigoplus^s_{j=1}\Z/p_j^{r_j}(i)
\]for a positive integer $m$ with prime factorization $m=p_1^{r_1}\cdots p_s^{r_s}$. We also put
\[
\Q/\Z(i):=\ilim_m\Z/m(i)=\bigoplus_p\Q_p/\Z_p(i),
\]where the direct sum is taken over all prime numbers. The cohomology functors and the unramified cohomology functors with coefficients in $\Q_p/\Z_p(i)$, $\Z/m(i)$ and $\Q/\Z(i)$ are defined in the natural way.

As a consequence of the Bloch--Kato conjecture, which has been proved by Bloch--Kato \cite{BlochKato86} and Gabber for $p=\car(F)$, and by Voevodsky \cite{Voe11BKconj} for $p\neq\car(F)$, for every $n\geqslant 1$ the natural maps
\[
\rH^n\big(K,\Z/p^r(n-1)\big)\longrightarrow \rH^n\big(K,\Q_p/\Z_p(n-1)\big) \quad
\text{and} \quad \rH^n\big(K,\Z/m(n-1)\big)\longrightarrow \rH^n\big(K,\Q/\Z(n-1)\big)
\]are injective for every field extension $K/F$, and hence the natural maps
\[
\rH^n_{nr}\big(X,\Z/p^r(n-1)\big)\longrightarrow \rH^n_{nr}\big(X,\Q_p/\Z_p(n-1)\big) \quad
\text{and} \quad \rH^n_{nr}\big(X,\Z/m(n-1)\big)\longrightarrow \rH^n_{nr}\big(X,\Q/\Z(n-1)\big)
\]are injective for every smooth connected $F$-variety $X$.

By purity for Brauer groups \cite{Cesnavicius19Duke}, the group $\rH^2_{nr}(X,\Q/\Z(1))$ coincides with the cohomological Brauer group $\Br(X)$ of $X$. \hfill $\blacksquare$

\

\newpara\label{3.3} Let $X$ be an algebraic variety over $F$. Let $\overline{F}$ be a separable closure of $F$ and put $\overline{X}=X\times_F\overline{F}$.
For each $i\in\N$, let $\CH^i(X)$ denote the Chow group of codimension $i$ cycles on $X$. There is a natural homomorphism
\[
\xi^i: \CH^i(X)\longrightarrow \CH^i(\overline{X})^{\mathrm{Gal}(\overline{F}/F)}
\]for each $i\in\N$.

Let us assume that $X=\mathrm{SB}_r(D)$ is the generalized Severi--Brauer variety of right ideals of reduced dimension $r\geqslant 1$ in a central simple algebra $D$ over $F$ (cf. \cite[(1.16)]{KMRT}). Then $X$ is a projective homogeneous variety in the sense of \cite{Merkurjev95AlgiAna} and \cite{Peyre95ProcSymPureMath58}, and $\CH^i(\ov{X})$ is a permutation Galois module (and in particular torsion-free as an abelian group). In this case, the kernel of $\xi^i$ coincides with the torsion subgroup $\CH^i(X)_{\mathrm{tors}}$ of $\CH^i(X)$.

Now consider the special case $r=1$; i.e., $X=\mathrm{SB}(D)$ is the usual Severi--Brauer variety of $D$. We recall the following facts about the Chow groups of $X=\mathrm{SB}(D)$.

\begin{enumerate}
  \item If $\ind(D)$ is not divisible by 8 nor by the square of any odd prime, then $\CH^i(X)$ is torsion-free for every $i\in\N$ (\cite[Cor.\;1.13 and Prop.\;1.15]{Merkurjev95ProcSympos58}).
    \item Suppose $D$ is a division algebra of square-free exponent $e$. Then we have $\CH^2(X)_{\mathrm{tors}}=0$ if for every prime divisor $p$ of $e$, the $p$-primary component of $D$ is decomposable into a tensor of two division algebras of smaller degree (\cite[Prop.\;5.3]{Karpenko98Ktheory}). \hfill $\blacksquare$
\end{enumerate}

In the next two theorems we gather together some known results that will be useful in this paper.

\begin{thm}\label{3.4}
  Let $D$ be a  central simple algebra over $F$ and let $X=\mathrm{SB}_r(D)$ be the generalized Severi--Brauer variety of right ideals of reduced dimension $r\geqslant 1$ in $D$. For each $n\geqslant 1$, consider the natural map
  \[
  \eta^n: \rH^n\big(F,\Q/\Z(n-1)\big)\longrightarrow \rH^n_{nr}\big(X,\,\Q/\Z(n-1)\big).
  \]
\begin{enumerate}
    \item For $n=1$, the map $\eta^1$ is an isomorphism.
    \item For $n=2$, the map $\eta^2$ is surjective and $\mathrm{Ker}(\eta^2)$ coincides with the subgroup generated by $r.(D)$, where $(D)\in \Br(F)=\rH^2(F,\Q/\Z(1))$ denotes the Brauer class of $D$.
    \item For $n=3$, there is an exact sequence
      \[
      F^*\xrightarrow{\cup\, r.(D)} \mathrm{Ker}(\eta^3)\longrightarrow \mathrm{CH}^2(X)_{\mathrm{tors}}\longrightarrow 0.
      \]
       \item Suppose that $r=1$ (i.e., $X=\mathrm{SB}(D)$ is the usual Severi--Brauer variety of $D$).

         If $\ind(D)=2$, then for every $n\geqslant 3$ we have
      \[
      \Ker(\eta^n)=\rH^{n-2}(F,\Q/\Z(n-2))\cup (D)\quad\text{ and } \quad \coker(\eta^n)=0.
      \]If $\ind(D)\,|\,4$ or $D$ is a decomposable division algebra of exponent $2$, then $\Ker(\eta^3)=F^*\cup (D)$.
       \item Suppose that $r=2$.

      If $\ind(D)\,|\,2$, then $\eta^n$ is an isomorphism for all $n\geqslant 1$.

     If $\exp(D)=2$ and $\ind(D)=4$, then $\eta^3$ is injective and $\coker(\eta^3)\cong \Z/2$.
\end{enumerate}
\end{thm}
\begin{proof}
  (1) This is a well known fact which is valid for all smooth proper geometrically rational varieties. See e.g. \cite[Prop.\;3.4]{HuSun23}.

  (2) As we have said in \eqref{3.2}, $\rH^2_{nr}(X,\Q/\Z(1))=\Br(X)$. We have
  \[
\Br(\ov{X})= \rH^2_{nr}\big(\ov{X},\Q/\Z(1)\big)\cong \rH^2(\ov{F},\Q/\Z(1))=0
  \]by \cite[Thm.\;8.6.1]{ColliotHooblerKahn97}, and since $\Pic(\ov{X})$ is a permutation Galois module, $\rH^1(F,\Pic(\ov{X}))=0$. So the surjectivity of $\eta^2$ follows from the exact sequence (5.21) in \cite[Prop.\;5.4.2]{ColliotSko21book}. The statement about $\Ker(\eta^2)$ is proved in \cite[Cor.\;2.7]{MerkurjevTignol95crelle}. (The case $r=1$ is a well known theorem due to Amitsur \cite[Thm.\;5.4.1]{GilleSzamuely17}).

  (3) This is  proved in \cite[\S\,3, Theorem]{Merkurjev95AlgiAna} and \cite[Thm.\;2.1]{Peyre98Ktheory}.

  (4) The stable birational class of $X$ depends only on the Brauer class of $D$ (\cite[Remark\;5.4.3]{GilleSzamuely17}), and the unramified cohomology group is a stable birational invariant (\cite[Thms.\;8.5.1 and 8.6.1]{ColliotHooblerKahn97}). So we may assume that $D$ is a division algebra.

  If $\ind(D)\,|\,4$ or $D$ is decomposable of exponent 2, then, as we have mentioned in \eqref{3.3}, $\CH^2(X)_{\mathrm{tors}}=0$. So the description of $\mathrm{Ker}(\eta^3)$ follows from the exact sequence in (3).

  Now suppose that  $\ind(D)=2$. Then $X$ is a smooth plane conic.

  If $\car(F)\neq 2$, the result has been noticed in \cite[Prop.\;7]{Berhuy07ArchMath}. If $\car(F)=2$, the description of $\Ker(\eta^n)$ follows from \cite[Lemma\;4.3 and Prop.\;5.1]{HuSun23}. The surjectivity of $\eta^n$ can be proved similarly to the proof of \cite[Prop.\;A.1]{KahnRostSujatha98}. (The case $n=3$ has been treated in \cite{HuSun23}.)

  Indeed, using (divisible coefficients analogs of) the exact sequences (5.4.1) and (5.4.2) in \cite{HuSun23}, we can get an exact sequence
  \[
  \begin{split}
   \rH^1_{\textrm{Zar}}\big(X,\sH^{n-1}(n-1)\big)&\overset{\phi}{\lra} \rH^0\left(F,\rH^1_{\textrm{Zar}}\big(\ov{X},\sH^{n-1}(n-1)\big)\right)=\rH^{n-2}\big(F,\Q/\Z(n-2)\big)\\
   &\lra \coker(\eta^n)\lra 0.
  \end{split}
  \]It is thus sufficient to show that the map $\phi$ is surjective. By composing $\phi$ with the cup product (cf. \cite[(3.3.2)]{HuSun23})
  \[
  \rH^{n-2}_{nr}(X,\Q/\Z(n-2))\otimes \rH^1_{\textrm{Zar}}\big(X,\sH^1(1)\big)\lra \rH^1_{\textrm{Zar}}\big(X,\sH^{n-1}(n-1)\big)
  \]and the natural map
  \[
 \rH^{n-2}(F,\Q/\Z(n-2))\otimes \rH^1_{\textrm{Zar}}\big(X,\sH^1(1)\big)\lra \rH^{n-2}_{nr}(X,\Q/\Z(n-2))\otimes \rH^1_{\textrm{Zar}}\big(X,\sH^1(1)\big)
  \]we obtain a natural map
  \[
  \phi':  \rH^{n-2}(F,\Q/\Z(n-2))\otimes \rH^1_{\textrm{Zar}}\big(X,\sH^1(1)\big)\lra \rH^{n-2}\big(F,\Q/\Z(n-2)\big).
  \]Here $\rH^1_{\textrm{Zar}}\big(X,\sH^1(1)\big)=\CH^1(X)\otimes(\Q/\Z)$, and $\CH^1(X)=2\Z$. So the map $\phi'$ can be identified with the multiplication by 2 on the divisible group
  $\rH^{n-2}\big(F,\Q/\Z(n-2)\big)$. Hence $\phi'$ is surjective, and it follows that $\phi$ is surjective, as desired.

(5)  By \cite[Prop.\;3]{Blankchet91}, the stable birational equivalence class of $X=\SB_2(D)$ depends only on the Brauer class of $D$. So we may assume that $D$ is a division algebra.

If $\ind(D)=2$, then we have $X=\SB_2(D)=\mathrm{Spec}(F)$, so that the assertion is obvious.

If $\exp(D)=2$ and $\ind(D)=4$, we may assume that $D$ is a biquaternion division algebra. In this case $X=\SB_2(D)$ is isomorphic to the projective quadric defined by the Albert form of $D$. So the result follows from \cite[Thm.\;2]{Kahn95ArchMath} and \cite[Thm.\;1.2]{HuSun23}.
\end{proof}


\begin{thm}\label{3.5}
  Let $F'/F$ be a finite separable extension and let $D'$ be a central simple algebra over $F'$. Let $X=\mathrm{R}_{F'/F}\SB(D')$ be the Weil restriction of the Severi--Brauer $F'$-variety of $D'$. For each $n\geqslant 1$, consider the natural map
  \[
  \eta^n: \rH^n\big(F,\Q/\Z(n-1)\big)\longrightarrow \rH^n_{nr}\big(X,\,\Q/\Z(n-1)\big).
  \]
\begin{enumerate}
    \item For $n=1$, the map $\eta^1$ is an isomorphism.

    \item For $n=2$, the map $\eta^2$ is surjective and $\mathrm{Ker}(\eta^2)$ coincides with the subgroup generated by the Brauer class $\Cor_{F'/F}(D')\in\Br(F)$, where $\Cor_{F'/F}$ denotes the corestriction map.
    \item For $n=3$, there is an exact sequence
      \[
      (F')^*\xrightarrow{\Cor_{F'/F}(\cdot\,\cup (D'))} \mathrm{Ker}(\eta^3)\longrightarrow \mathrm{CH}^2(X)_{\mathrm{tors}}\longrightarrow 0.
      \]
    \item If $\ind(D')\,|\,2$, then $\mathrm{Ker}(\eta^3)=\Cor_{F'/F}((F')^*\cup(D'))$.
\end{enumerate}
\end{thm}
\begin{proof}
   The statement about $\Ker(\eta^2)$ follows from \cite[Cor.\;2.12]{MerkurjevTignol95crelle}.   The other statements in (1)--(3) hold for the same reasons  (by the same references) as in the corresponding statements in Thm.\;\ref{3.4}.

   If $D'$ is a quaternion algebra over $F'$, then $X$ is a 2-dimensional projective homogeneous variety. In this case, the degree map $\CH^2(X)=\CH_0(X)\to\Z$ is injective by \cite[Prop.\;4.4]{ChernousovMerkurjev06ComposMath}, so $\CH^2(X)_{\mathrm{tors}}=0$. Thus (4) follows from the exact sequence in (3).
\end{proof}

\newpara\label{3.6} Let $D$ be a central simple algebra over $F$. The function fields $F_+(D)=F(\SB(D))$ and $F_-(D)=F(\SB_2(D))$ in Definition\;\ref{2.14} are generic index reduction fields of orthogonal and symplectic types respectively. They are also nice extensions of $F$ in the sense defined in \eqref{3.1}. So we can define for each $n\geqslant 1$ the unramified cohomology groups
\[
\rH^n_+(D):=\rH^n_{nr}\big(F_+(D)/F,\,\Z/2(n-1)\big)\quad\text{and}\quad \rH^n_-(D):=\rH^n_{nr}\big(F_-(D)/F,\,\Z/2(n-1)\big).
\]For any generic index reduction field of orthogonal type $K_+$ (resp. of symplectic type $K_-$) for the Brauer class $(D)\in\Br(F)$ that is a nice extension of $F$, we have natural isomorphisms
\[
\rH^n_+(D)\cong\rH^n_{nr}\big(K_+/F,\,\Z/2(n-1)\big)\quad\text{and}\quad \rH^n_-(D)\cong\rH^n_{nr}\big(K_-/F,\,\Z/2(n-1)\big),
\]because the nice models of $K_+$ (resp. of $K_-$) are all stably birational to $\SB(D)$ (resp. $\SB_2(D)$). In particular, $\rH^n_+(D)$ and $\rH^n_-(D)$ depend only on the Brauer class of $D$.

Similarly, for a central simple algebra $D'$ over a quadratic separable extension $F'$ of $F$, we can define
\[
\rH^n(D'/F):=\rH^n_{nr}\big(K/F,\,\Z/2(n-1)\big),
\]where $K$ is any nice extension of $F$ that is a generic index reduction field of $F$-unitary type for $(D')$ (e.g., $K=F(D')$ as in Definition\;\ref{2.14}).

We also define the group $\rH^n(D'/F)$ when $D'\cong D\times D^{op}$ for some central simple $F$-algebra $D$, where $D^{op}$ denotes the opposite algebra of $D$. In this case, we put
\[
\rH^n(D'/F):=\rH^n_{nr}\big(K/F,\,\Z/2(n-1)\big),
\]where $K$ is any nice extension of $F$ that is a generic splitting field of $(D)$. \hfill $\blacksquare$

\subsection{Cohomological invariants of quadratic forms}

Unless otherwise stated, we follow \cite{EKM08} for terminology and notation about quadratic forms. For the reader's convenience,
we recall some terminology and facts that will be used in this paper.

\

\newpara\label{3.7} For a quadratic form $(V,q)$ over $F$, let $b_q$ denote the bilinear form associated to $q$ given by
$b_q(u,v)=q(u+v)-q(u)-q(v)$. Put
\[
\rad(b_q)=\{v\in V\,|\,b_q(u,v)=0 \text{ for all } u\in V\}.
\]Following the terminology of \cite{KMRT} and \cite{EKM08}, we say that $b$ is nondegenerate if either $\rad(b_q)=0$ or $\rad(b_q)$ is 1-dimensional generated by an anisotropic vector for $q$.

We write $\rI_q(F)$ or $\rI_q^1(F)$ for the Witt group of even-dimensional nondegenerate quadratic forms over $F$. For $n\geqslant 2$, let $\rI^n_q(F)$ denote the subgroup of $\rI_q(F)$ generated by the quadratic $n$-fold Pfister forms over $F$. For a quadratic form $\varphi$ over $F$, we will write $\varphi\in \rI^n_q(F)$ if $\varphi$ is nondegenerate, of even dimension, and its Witt class lies in $\rI^n_q(F)$. \hfill $\blacksquare$

\

\newpara\label{3.8} Let us introduce the shorthand notation
\[
\rH^n(\cdot)=\rH^n\big(\cdot,\Z/2(n-1)\big)
\] for every $n\geqslant 1$. The group $\rH^1(F)$ classifies quadratic \'etale $F$-algebras. Any such algebra can be represented by $F_c$, where
\[
F_c=F[T]/(T^2-c) \quad \text{ with } c\in F^* \text{ if } \car(F)\neq 2
\]and
\[
F_c=F[T]/(T^2-T-c) \quad \text{ with } c\in F \text{ if } \car(F)=2.
\]The cohomology class in $\rH^1(F)$ corresponding to the quadratic algebra $F_c/F$ will be denoted by $(c\,]$. The binary quadratic form given by the norm of $F_c/F$ will be denoted by $\langle\!\langle c\,]\!]$.

The map
\[
e_1:\; \rI^1_q(F)\longrightarrow \rH^1(F)\,;\quad \langle\!\langle a]\!]\longmapsto (a\,]
\]is a well defined homomorphism,  often called the \emph{discriminant}. (In characteristic 2 it is also called the \emph{Arf invariant}.) It is well known that $e_1$ is surjective with $\mathrm{Ker}(e_1)=\rI^2_q(F)$.

For $n\geqslant 2$, as explained in \cite[\S\,16]{EKM08}, one can also define a functorial homomorphism
\[
e_n:\;\rI^n_q(F)\longrightarrow \rH^n(F),
\]which is surjective with $\mathrm{Ker}(e_n)=\rI^{n+1}_q(F)$. Using the same notation as in \cite{EKM08}, one can characterize $e_n$ by the formula
\[
e_n\bigl(\langle\!\langle a_1,\cdots, a_{n-1}\,;\,a_n]\!]\bigr)=(a_1)\cup \cdots\cup (a_{n-1})\cup (a_n\,]\,,
\]where for any $a\in F^*$, $(a)$ denotes its canonical image in $F^*/F^{*2}$. (The maps $e_2$ and $e_3$ are more classical, called the \emph{Clifford invariant} and the \emph{Arason invariant} respectively.)

Note that for all $n\geqslant 1$, we have
\[
  e_n(\varphi)=e_n(c.\varphi)\quad\text{for all } \varphi\in \rI^n_q(F)\quad\text{and all}\quad c\in F^*,
\]
because $\varphi-c\varphi=\langle\!\langle c\rangle\!\rangle\varphi\in \rI^{n+1}_q(F)=\mathrm{Ker}(e_n)$. \hfill $\blacksquare$

\

\newpara\label{3.9} As a consequence of the Hauptsatz, it is well known that $\bigcap_{n\geqslant 1}\rI^n_q(F)=0$ (\cite[Cor.\;23.8]{EKM08}). Therefore, an even dimensional nondegenerate quadratic form $q$ over $F$ is hyperbolic if and only if its Witt class lies in $\rI^n_q(F)$ for all $n\geqslant 1$, or equivalently, if and only if $e_n(q)=0$ for all $n\geqslant 1$.

Let $q$ and $q'$ be nondegenerate quadratic forms of the same dimension over $F$. We know that $q$ and $q'$ are isomorphic if and only if they are Witt equivalent. If $\car(F)\neq 2$ or their common dimension is even, then the orthogonal sum $q\perp -q'$ is nondegenerate, so that $e_1(q\perp -q')$ can be defined and whenever $e_n(q\perp -q')=0$, the next invariant $e_{n+1}(q\perp -q')$ is well defined. Therefore, when $\car(F)\neq 2$ or the common dimension $\dim q=\dim q'$ is even, we can use the $e_n$-invariants of $q\perp -q'$ for all $n\geqslant 1$ to determine whether $q$ and $q'$ are isomorphic. \hfill $\blacksquare$

\section{Cohomological invariants in symplectic and unitary cases}

\subsection{Symplectic case}

\newpara\label{4.1} Let $(Q,\gamma)$ be a quaternion algebra with its canonical (symplectic) involution over a field $K$. We have
 \[
 \Symd(Q,\gamma)=\{a+\gamma(a)\,|\,a\in Q\}=K.
 \]Let $(V,h)$ be an even hermitian form over $(Q,\gamma)$. Then the \emph{Jacobson trace form} of $h$ is the quadratic form $q_h$ defined on the underlying $K$-vector space of $V$ by the rule $q_h(x)=h(x,x)$. Note that
 \[
 \dim_K(q_h)=\dim_K(V)=\deg(Q)\ind(Q)\rk(h)=2\,\ind(Q)\rk(h).
 \]

 A standard fact is that the hermitian form $h$ is hyperbolic if and only if the quadratic form $q_h$ is hyperbolic, and that the isomorphism class of $q_h$ determines that of $h$ (\cite[Chap.\;10, Thm.\;1.7]{Schar}). \hfill $\blacksquare$

\

\newpara\label{4.2} Let $D$ be a central simple algebra over a field $F$ and let $\theta$ be a symplectic involution on $D$. (The existence of a symplectic involution on $D$ implies that $\deg(D)$ is even and that $\exp(D)\leqslant2$.) Let $K$ be a nice extension of $F$ that is a generic index reduction field of sympletic type of $D$, e.g., $K=F_-(D)$ as in Definition\;\ref{2.14}. Then $D_K$ is Brauer equivalent to a quaternion algebra $Q$ over $K$, and we may assume $D_K=\End_Q(W)$ for a finitely generated right $Q$-module $W$. Let $\gamma$ be the canonical involution on $Q$. By \eqref{2.5}, there is a nondegenerate hermitian form $f$ on $W$ (with respect to $\gamma$) such that $\theta_K=\ad_{f/\gamma}$. The form $f$ is uniquely determined up to a scalar multiple from $K^*$. We fix such a form $f$.

Let $(V,h)$ be a nondegenerate even hermitian form over $(D,\theta)$. To the hermitian form $(V_K,h_K)$ over $D_K$ obtained by base extension, the Morita functor (cf. \eqref{2.4}) associates a nondegenerate even hermitian form $h^{\fm}_K=f\star h_K$ over $(Q,\gamma)$. Let $q^{\fm}_h$ denote the Jacobson trace form of $h^{\fm}_K$. It is uniquely determined by $h$ itself once $f$ is fixed, and a different choice of $f$ modifies $h$ only by a scalar multiple in $K^*$.

Recall from \eqref{2.3} that the rank of a hermitian form is the semisimple rank of the underlying module. The Morita functor does not alter the rank of hermitian forms. So we have
\[
\rk(h^{\fm}_K)=\rk(h_K)=\frac{\ind(D)}{\ind(D_K)}\rk(h)
\]and hence
\[
\dim_K(q^{\fm}_h)=2\,\ind(Q)\rk(h^{\fm}_K)=2\,\ind(Q)\frac{\ind(D)}{\ind(D_K)}\rk(h)=2\,\ind(D)\rk(h).
\]In particular, the quadratic form $q_h^{\fm}$ is always even dimensional.

For every integer $n\geqslant 1$, we define the $e_n$-\emph{invariant} of $h$ by
\[
e_n(h):=e_n(q^{\fm}_h)\;\in\; \rH^n(K).
\]Here for $n\geqslant 2$, the invariant $e_n(h)$ is defined only when $e_{n-1}(h)=0$. Since the $e_n$-invariants of quadratic forms are insensitive to scalar multiples, these invariants of $h$ depend only on the involution $\theta$ and are independent of the choice of $f$.

In the above discussions the algebras $D/F$ and $Q/K$ are not required to be division algebras.  Since the Morita functor commutes with extensions of base fields, this allows us to check easily that the above $e_n$-invariants are functorial in the base field. That is, for any field extension  $L/F$ we have
\[
\mathrm{Res}_{LK/K}(e_3(h))=e_3(h_L)\in \rH^3(LK),
\]where $LK$ denote the free composite of $L$ and $K$, and
\[
\mathrm{Res}_{LK/K}:\; \rH^3(K)\lra  \rH^3(LK)
\]denotes the natural restriction map. \hfill $\blacksquare$

\begin{lemma}\label{4.3}
  With notation and hypotheses as in $\eqref{4.2}$, we have
  \[
e_n(h)\in \rH^n_-(D)=\rH^n_{nr}\big(K/F,\,\Z/2(n-1)\big) \quad \text{ for every }\; n\geqslant 1.
  \]
\end{lemma}
\begin{proof}
Let $X$ be a nice model of $K/F$ and let $x$ be any codimension 1 point of $X$. Then the local ring $\mathscr{O}_x$ of $X$ at $x$ is a discrete valuation ring containing $F$. So the hermitian form $h_K$ over the function field $K=F(X)$ can be obtained by first base changing $h$ to $\mathscr{O}_x$ and then to $K$. If $\widehat{\mathscr{O}}_x$ denotes the completion of $\mathscr{O}_x$ and $K_x$ denotes the fraction field of $\widehat{\mathscr{O}}_x$, then  the quadratic form $q^{\mathfrak{m}}_{K_x}$ over the field $K_x$ is defined over the valuation ring $\widehat{\mathscr{O}}_x$. This implies that $e_n(q_K^{\mathfrak{m}})$ is unramified at $x$ (cf. \cite[\S\,7]{HuSun23}).
\end{proof}

\newpara\label{4.4} Retain the notation and hypotheses of \eqref{4.2}.

If $D$ is split, then every nondegenerate even hermitian form $h$ over $(D,\theta)$ is hyperbolic, so $e_n(h)=0$ for all $n\geqslant 1$.

Suppose that $D$ is not split. Then we have $\ind(D_K)=2$ by \cite[Thm.\;3]{Blankchet91}. The reduced norm $\Nrd_Q$ of the quaternion $K$-algebra $Q$ is thus an anisotropic 2-fold Pfister form over $K$. The hermitian form $h^{\fm}_K$ can be put into a diagonal form $\langle \lambda_1,\cdots,\lambda_{dr}\rangle_{\mathrm{h}}$, where $r=\rk(h)$, $d=\ind(D)/2$ and $\lambda_i\in K^*$. (The subscript ``$\mathrm{h}$'' here indicates that the form under consideration is a hermitian form.) Then we have
\[
q^{\fm}_h=\langle \lambda_1,\cdots,\lambda_{dr}\rangle_{\mathrm{bil}}\otimes\Nrd_Q\,\in\;\rI_q^2(K).
\](The subscript ``$\mathrm{bil}$'' here indicates that the form under consideration is a bilinear form.) It follows that
\begin{equation}\label{4.4.1}
  \begin{split}
    e_1(h)&=e_1(q^{\fm}_h)=0,\,\\
    e_2(h)&=e_2(q^{\fm}_h)=dr.(Q)=dr.(D_K)=\begin{cases}
  (D_K) \quad & \text{ if }\; dr \text{ is odd},\\
  0 \quad & \text{ if }\; dr \text{ is even},\\
\end{cases}\\
e_3(h)&=e_3(q^{\fm}_h)=\big((-1)^{\frac{dr}{2}}\lambda_1\cdots\lambda_{dr}\big)\cup (D_K)\quad (\text{when } dr \text{ is even})\,.
  \end{split}
\end{equation}
Since $2(D)=0$, the natural map
\[
\eta^2: \Br(F)=\rH^2(F,\Q/\Z(1))\longrightarrow \rH^2_{nr}(X,\Q/\Z(1))=\Br(X)
\]for $X=\SB_2(D)$ is an isomorphism, by Thm.\;\ref{3.4} (2). So we see from \eqref{4.4.1} that the invariant $e_2(h)$ always lies in the 2-torsion subgroup
\[
\Br(F)[2]=\rH^2(F,\Q/\Z(1))[2]\cong \rH^2_{nr}(X,\Q/\Z(1))[2]=\rH^2_-(D),
\] and when $D$ is not split, $e_2(h)$ vanishes if and only if  $dr=\frac{\ind(D)}{2}\rk(h)=\frac{1}{2}\deg(\End_D(V))$ is even.

If $\ind(D)\leqslant2$, then the natural map
\[
\eta^n: \rH^n(F,\Q/\Z(n-1))\longrightarrow \rH^n_{nr}(X,\Q/\Z(n-1))
\] is an isomorphism by Thm.\;\ref{3.4} (5). So in this case, $e_n(h)=e_n(q_h)$ lies in the cohomology group $\rH^n(F)$ of the base field $F$ for every $n\geqslant 1$, by Lemma\;\ref{4.3}. This shows that when $\ind(D)\leqslant2$, our definition of $e_n$ recovers the classical construction (cf. \cite[\S\,5.1]{Tignol10}).

Finally, note that if $\ind(D)=4$ and if $e_3(h)$ descends to a cohomology class in $\rH^3(F)$, then it descends uniquely by Thm.\;\ref{3.4} (5). \hfill $\blacksquare$

\begin{thm}\label{4.5}
Let $(D,\theta)$ be a CSIA of symplectic type over $F$ and let $h$ be a nondegenerate even hermitian form over $(D,\theta)$.

If $\car(F)\neq 2$, then $h$ is hyperbolic if and only if $e_n(h)=0$ for every $n\geqslant 1$.
\end{thm}
\begin{proof}
Let $K$ be a generic index reduction field of symplectic type of $D$. By the symplectic case of Theorem\;\ref{2.15}, $h$ is hyperbolic if $h_K$ is hyperbolic. The hyperbolicity of $h_K$ is equivalent to that of the quadratic form $q_h^{\fm}$ over $K$. So the result follows from the well known case of quadratic forms, as we have recalled in \eqref{3.9}.
\end{proof}

The following result proves the Morita invariance of the above $e_n$-invariants.

\begin{prop}\label{4.6}
  Let $(D,\theta)$ be a CSIA of symplectic type over $F$ and let $E=\End_D(V_0)$ for some finitely generated $D$-module $V_0$. Let $h_0$ be a nondegenerate even $\theta$-hermitian form on $V_0$ and put $\tau=\ad_{h_0/\theta}$.

  Then for any nondegenerate even $\tau$-hermitian form $g$ over $(E,\tau)$, letting $h=h_0\star g$ be the hermitian form over $(D,\theta)$ obtained from $g$ by applying the Morita functor, we have $e_n(g)=e_n(h)$ for all $n\geqslant 1$.
\end{prop}
\begin{proof}
 With notation as in \eqref{4.2}, let $K$ be a nice extension of $F$ that is a generic index reduction field of symplectic type of the Brauer class $(D)=(E)\in\Br(F)$. Then $E_K$ and $D_K$ are Brauer equivalent to the same quaternion $K$-algebra $Q$. Moreover, we can choose the Morita functor from $(E_K,\tau_K)$ to $(Q,\gamma)$ to be the one associated to the hermitian form $f\star (h_0)_K$, because
  \[
\tau_K=\ad_{h_0/\theta_K}=\ad_{(f\star h_0)/\gamma}.
  \]Therefore, the quadratic $K$-forms $q_g^{\fm}$ and $q_h^{\fm}$ obtained from $g$ and $h$ are the same. So the result follows from the definition.
\end{proof}

\newpara\label{4.7} As in \eqref{4.2}, let $(D,\theta)$ be a CSIA of symplectic type over $F$ and let $K$ be a nice extension of $F$ that is a generic index reduction field of symplectic type. Recall from \eqref{2.3} that $\rW_+(D,\theta)$ denotes the Witt group of even hermitian forms over $(D,\theta)$. Since the form $f$ to realize $\theta_K$ as an adjoint involution is fixed for all hermitian forms over $(D,\theta)$, it is clear from the construction that the $e_1$-invariant defines a group homomorphism
\[
e_1: \rW_+(D,\theta)\longrightarrow \rH^1_-(D).
\]We define
\[
\rI^1(D,\theta):=\rW_+(D,\theta),
\]and
\[
\rI^n(D,\theta):=\{[h]\in \rW_+(D,\theta)\,|\, e_1(h)=\cdots=e_{n-1}(h)=0\}
\] for each $n\geqslant 2$. Then we have functorial group homomorphisms
\[
e_n:\; \rI^n(D,\theta)\lra \rH^n_-(D).
\]These homomorphisms are compatible with Morita equivalence, by Prop.\;\ref{4.6}. \hfill $\blacksquare$

\medskip

\newpara\label{4.8} Let $A$ be a central simple $F$-algebra of even degree $2m$. We assume that $A$ has exponent $\exp(A)\leqslant2$, so that there exist symplectic involutions on $A$.
Then we can choose a central simple $F$-algebra $D$ of even degree such that $A=\End_D(V)$ for a finitely generated right $D$-module $V$. Note that the semisimple rank $\ssrk_D(V)$ of $V$ over $D$ is equal to the coindex $\coind(A)=\deg(A)/\ind(A)$ of $A$.

Let us also choose a symplectic involution $\theta$ on $D$. Then for any symplectic involution $\sigma$ on $A$, there is a nondegenerate even $\theta$-hermitian form $h$ on $V$ such that $\sigma=\ad_{h/\theta}$. This form $h$, of rank $\rk(h)=\coind(A)$, is uniquely determined up to a scalar multiple from the base field $F$. So we can define the $e_n$-\emph{invariant} of $\sigma$ for each $n\geqslant 1$ by
\[
e_n(\sigma):=e_n(h)\in \rH^n_-(D)=\rH^n_-(A).
\](For $n\geqslant 2$, $e_n(\sigma)$ is defined only when $e_{n-1}(\sigma)=0$.) By Prop.\;\ref{4.6}, this definition is independent of the choice of $(D,\theta)$.

By Thm.\;\ref{4.5}, the symplectic involution $\sigma$ is hyperbolic if and only if $e_n(\sigma)=0$ for all $n\geqslant 1$.

Since the split case is trivial, let us now suppose that $A$ is not split. Then $\ind(A)=\ind(D)$ is even. As we have seen in \eqref{4.4}, the following statements hold:

\begin{enumerate}
   \item The invariant $e_1(\sigma)$ is always trivial.
  \item The invariant $e_2(\sigma)$ descends uniquely to the base field $F$, and it vanishes precisely when
  \[
  \frac{1}{2}\coind(A)\ind(A)=\frac{1}{2}\deg(A)
  \] is even.

  So $e_3(\sigma)$ is defined only when $\deg(A)\equiv 0\pmod{4}$.
  \item If $\ind(A)=2$, then $e_n(\sigma)$ (when it is defined) descends uniquely to $F$ for every $n\geqslant 1$.
  \item If $\ind(A)=4$ and $e_3(\sigma)$ descends to $F$, then it descends uniquely. \hfill $\blacksquare$
\end{enumerate}

\newpara\label{4.9} With notation as in \eqref{4.8}, assume $\deg(A)\equiv 0\pmod{4}$. Let $\sigma$ and $\tau$ be symplectic involutions on $A$. There exists $s\in\Symd(A,\sigma)^\times:=\{\sigma(x)+x\,|\,x\in A\}\cap A^\times$ such that $\tau=\Int(s)\circ \sigma$, where $A^\times$ denotes the group of units in $A$ and $\Int(s):A\to A$ denotes the inner automorphism $x\mapsto sxs^{-1}$. As shown in \cite{BerhuyMonsurroTignol03}, using the Pfaffian reduced norm map $\Nrp_{\sigma}:\Symd(A,\sigma)\to F$ one can define the \emph{relative Rost invariant}
\[
\sR_{\sigma}(\tau):=(\Nrp_\sigma(s))\cup (A)\in \rH^3(F).
\]This invariant  is denoted by $\Delta_{\sigma}(\tau)$ and called the \emph{discriminant} of $\tau$ \emph{with respect to} $\sigma$ in \cite{BerhuyMonsurroTignol03} and \cite{GaribaldiParimalaTignol09}. 
(The papers \cite{BerhuyMonsurroTignol03} and \cite{GaribaldiParimalaTignol09} assume the base field to have characteristic different from 2, but as the reader can easily check, all the definitions and results in those papers that are related to what we discuss in the present paper are still valid in characteristic 2.)

Let $K$ be a nice extension of $F$ that is a generic index reduction field of symplectic type for the Brauer class $(A)$, e.g., $K=F_-(A)$ as in Definition\;\ref{2.14}. Then
\[
\coind(A_K)=\frac{\deg(A_K)}{\ind(A_K)}=\frac{\deg(A)}{2}
\]is even. So there exists a hyperbolic symplectic involution $*$ on $A_K$, and we have
\begin{equation}\label{4.9.1}
  \Res_{K/F}(\sR_\sigma(\tau))=\sR_{\sigma_K}(\tau_K)=\sR_*(\tau_K)-\sR_*(\sigma_K)\in \rH^3(K).
\end{equation}
Let $g$ and $h$ be $\theta$-hermitian forms  on $V$ such that $\sigma=\ad_{g/\theta}$ and $\tau=\ad_{h/\theta}$. Then $\sigma_K$ and $\tau_K$ are adjoint to the hermitian forms  $g_K^{\fm}$ and $h^{\fm}_K$ over $Q$ respectively. Thus, by \eqref{4.4.1} and \cite[p.203, Example\;2]{BerhuyMonsurroTignol03}, we have
\begin{equation}\label{4.9.2}
e_3(h)=\sR_*(\tau_K)\quad \text{and}\quad e_3(g)=\sR_*(\sigma_K).
\end{equation}This combined with \eqref{4.9.1} shows that
\begin{equation}\label{4.9.3}
  \Res_{K/F}(\sR_\sigma(\tau))=e_3(h)-e_3(g)=e_3(\tau)-e_3(\sigma)\in \rH^3(K).
\end{equation}As a consequence, we see that, although $e_3(\tau)$ and $e_3(\sigma)$ may not descend to the base field $F$ (cf. \cite[\S\,5.3]{Tignol10}), the difference $e_3(\tau)-e_3(\sigma)$ always descends. In particular, if $\coind(A)$ is even so that $\sigma$ can be chosen hyperbolic, then $e_3(\tau)$ descends to $F$.

If we further assume $\frac{1}{2}\deg(A)\equiv 0\pmod{4}$, then an absolute invariant $\sR(A,\tau)=\Delta(A,\tau)\in \rH^3(F)$ was defined in \cite{GaribaldiParimalaTignol09}. Here again, one can check that all the results we cite from that paper have analogs in characteristic 2. Indeed, when $A$ has symbol length $\geqslant 3$ and $Y$ denotes the Albert quadric associated to the biquaternion algebra $Q_1\otimes Q_2$ as in \cite[p.353]{GaribaldiParimalaTignol09}, we can prove that the cohomology class $\Delta(A,\tau)_{F(Y)}$ is unramified by using an argument similar to the proof of Lemma\;\ref{4.3}, instead of the residue map argument in \cite{GaribaldiParimalaTignol09}. The fact that in characteristic 2 the class $\Delta(A,\tau)_{F(Y)}$ still descends uniquely to $F$ can be shown by using  results from \cite[\S\,8]{HuSun23}.

By the functoriality of $\sR(A,\tau)$ and \eqref{4.9.2}, we have
\begin{equation}\label{4.9.4}
  \Res_{K/F}(\sR(A,\tau))=\sR_*(\tau_K)=e_3(\tau_K)=e_3(h).
\end{equation}

The formulas \eqref{4.9.3} and \eqref{4.9.4} suggest that up to base change to the generic index reduction field $K=F(X)$ our invariant $e_3(\tau)$ can be viewed as an ``absolute'' version of the relative Rost invariant $\sR_\sigma(\tau)$ defined previously in \cite{BP2} and \cite{BerhuyMonsurroTignol03}, and that it is on the other hand compatible with the absolute Rost invariant $\sR(A,\tau)$ defined in \cite{GaribaldiParimalaTignol09} (when $\deg(A)\equiv 0\pmod{8}$). \hfill $\blacksquare$

\

\newpara\label{4.10} With notation as in \eqref{4.8}, now drop the assumption $\deg(A)\equiv 0\pmod{4}$ and only assume $\deg(A)$ is even. Let $g$ and $h$ be nondegenerate even $\theta$-hermitian forms on the same $D$-module $V$, and let $\sigma=\ad_{g/\theta}$ and $\tau=\ad_{h/\theta}$ be their adjoint involutions on $A=\End_D(V)$.

Since $e_2(-g\perp h)=0$, we can still define the \emph{relative $e_3$-invariant}
\[
e_3(h/g):=e_3(-g\perp h)\in \rH^3_{nr}\big(X,\Q/\Z(2)\big)\subseteq \rH^3\big(K,\Q/\Z(2)\big).
\]For any $\lambda\in F^*$, we have
\[
e_3(-g\perp \lambda h)-e_3(-g\perp h)=e_3(-h\perp\lambda h)=(\lambda^{\deg(A)})\cup (D_K)=\deg(A).(\lambda)\cup(D_K)=0
\]by \eqref{4.4.1}. Similarly, $e_3(-\lambda g\perp h)-e_3(-g\perp h)=0$. Therefore, the \emph{relative $e_3$-invariant} $e_3(\tau/\sigma)$ of the involutions $\sigma$ and $\tau$ can be defined as
\[
e_3(\tau/\sigma):=e_3(h/g)=e_3(-g\perp h)\in \rH^3_{nr}\big(X,\Q/\Z(2)\big)\subseteq \rH^3\big(K,\Q/\Z(2)\big).
\]

If $\deg(A)\equiv 0\pmod{4}$, then $e_3(\tau/\sigma)=\Res_{K/F}(\sR_\sigma(\tau))$ by \eqref{4.9.3}.

If $\deg(A)\equiv 2\pmod{4}$, then $\ind(A)=\ind(D)\leqslant2$. In this case, $e_3(\tau/\sigma)=e_3(h/g)$ descends uniquely to the base field $F$ because
\[
 \rH^3\big(F,\Q/\Z(2)\big)\cong \rH^3_{nr}\big(X,\Q/\Z(2)\big),
\] as was mentioned in \eqref{4.4}. \hfill $\blacksquare$

\subsection{Unitary case}

\newpara\label{4.11} Let $K'$ be a quadratic separable algebra over a field $K$. The nontrivial element $\gamma$ of the $K$-algebra automorphism group $\mathrm{Aut}_K(K')$ is an involution on $K'$. For any  hermitian form $(V,h)$  over $(K',\gamma)$, its \emph{Jacobson trace form} is the quadratic form $q_h$ defined on the underlying $K$-vector space of $V$ by the rule $q_h(x)=h(x,x)$.

It is well known that the hermitian form $h$ is hyperbolic if and only if the quadratic form $q_h$ is hyperbolic, and that the isomorphism class of $q_h$ determines that of $h$ (\cite[Chap.\;10, Thm.\;1.1]{Schar}).

If $K'$ is not a field, then every nondegenerate hermitian form (of constant rank) over $(K',\gamma)$ is hyperbolic. \hfill $\blacksquare$

\

\newpara\label{4.12} Let $(D,\theta)$ be a CSIA of unitary type over a field $F$. Let  $F'=Z(D)$ be the center of $D$. If $F'$ is a field, let $K$ be a nice extension of $F$ that is a generic index reduction field of $(D,\theta)$. If $F'\cong F\times F$, then $D\cong D_0\times D_0^{op}$ for some central simple $F$-algebra $D_0$. In this case, let $K$ be a nice extension of $F$ that is a generic splitting field of $D_0$.

  The base extension $(D_K,\theta_K):=(D\otimes_FK,\,\theta\otimes\mathrm{id}_K)$ is a CSIA of unitary type over $K$. Put $K'=K\otimes_FF'$. We may write $D_K=\End_{K'}(W)$ for a finitely generated $K'$-module  $W$ of constant semisimple rank. Let $\gamma$ be the nontrivial element of the $K$-algebra automorphism  group $\mathrm{Aut}_K(K')$. By \eqref{2.5}, there is a nondegenerate hermitian form $f$ on $W$ (with respect to $\gamma$) such that $\theta_K=\ad_{f/\gamma}$. The form $f$ is uniquely determined up to a scalar multiple from $K^*$. We fix such a form $f$.

Let $(V,h)$ be a nondegenerate hermitian form over $(D,\theta)$. To the hermitian form $(V_K,h_K)$ over $(D_K,\theta_K)$ obtained by base extension, the Morita functor (cf. \eqref{2.4}) associates a nondegenerate  hermitian form $h^{\fm}_K=f\star h_K$ over $(K',\gamma)$. Let $q^{\fm}_h$ denote the Jacobson trace form of $h^{\fm}_K$. It is uniquely determined by $h$ itself once $f$ is fixed, and a different choice of $f$ modifies $h$ only by a scalar multiple in $K^*$. We have
\[
\rk(h^{\fm}_K)=\rk(h_K)=\frac{\ind(D)}{\ind(D_K)}\rk(h)=\ind(D)\rk(h)
\]and hence
\[
\dim_K(q^{\fm}_h)=2\,\rk(h^{\fm}_K)=2\,\rk(h)=2\,\ind(D)\rk(h).
\]In particular, the quadratic form $q_h^{\fm}$ is always even dimensional.

For every integer $n\geqslant 1$, we define the $e_n$-\emph{invariant} of $h$ by
\[
e_n(h):=e_n(q^{\fm}_h)\;\in\; \rH^n(K).
\]Here for $n\geqslant 2$, the invariant $e_n(h)$ is defined only when $e_{n-1}(h)=0$. As in the symplectic case discussed in \eqref{4.2}, these invariants of $h$ depend only on the involution $\theta$ and are independent of the choice of $f$. From the definition, it is also clear that these invariants are functorial in the base field $F$ in the following sense: For every field extension $L/F$, we have
\[
\Res_{LK/K}(e_n(h))=e_n(h_L)\in \rH^n(LK)
\]whenever $e_n(h)$ is defined. \hfill $\blacksquare$

\begin{lemma}\label{4.13}
  With notation and hypotheses as in $\eqref{4.12}$, we have
  \[
e_n(h)\in \rH^n(D/F)=\rH^n_{nr}\big(K/F,\,\Z/2(n-1)\big) \quad \text{ for every }\; n\geqslant 1.
  \]
\end{lemma}
\begin{proof}This is proved in the same way as in the proof of Lemma\;\ref{4.3}.
\end{proof}

\newpara\label{4.14} Retain the notation and hypotheses of \eqref{4.12}. We have $\Cor_{F'/F}(D)=0$ in $\Br(F)$ since $D$ admits a unitary $F'/F$-involution. So, by Thm.\;\ref{3.5} (1) and (2),
\begin{equation}\label{4.14.1}
  \rH^n(D/F)=\rH^n(F)\text{  for  } n=1,2\,.
\end{equation} Therefore, the invariants $e_1(h)$ and $e_2(h)$ descend uniquely to cohomology classes defined over the base field $F$.

Let $(c\,]\in \rH^1(F)$ correspond to $F'/F$. That is, $c\in F$ is chosen such that
\[
F'=F_c=\begin{cases}
  F[T]/(T^2-c) \quad & \text{ if }\; \car(F)\neq 2,\\
  F[T]/(T^2-T-c)\quad & \text{ if }\;\car(F)=2\,.
\end{cases}
\] Put $h^{\fm}_K$ into a diagonal form $h^{\fm}_K=\langle \lambda_1,\cdots,\lambda_{dr}\rangle_{\mathrm{h}}$, where $d=\ind(D)$, $r=\rk(h)$ and $\lambda_i\in K^*$. Then we have
\[
q^{\fm}_K=\langle \lambda_1,\cdots,\lambda_{dr}\rangle_{\mathrm{bil}}\otimes \langle\!\langle c\,]\!]\;\in\;\rI_q(K).
\]Hence,
\begin{equation}\label{4.14.2}
\begin{split}
   e_1(h)&=e_1(q^{\fm}_K)=dr.(c\,]=\begin{cases}
     (c\,] \quad & \text{ if }\; dr \text{ is odd},\\
  0\quad & \text{ if }\; dr \text{ is even},
   \end{cases}\\
   e_2(h)&=\big((-1)^{\frac{dr}{2}}\lambda_1\cdots\lambda_{dr}\big)\cup (c\,] \quad (\text{when } (c\,]=0 \text{ or } dr \text{ is even}).
\end{split}
\end{equation}

Suppose that $\deg(\End_D(V))=\ind(D)\rk(h)=dr$ is even. Let $\cD(h)$ be the discriminant algebra of the adjoint involution $\ad_{h/\theta}$ on $\End_D(V)$. By \cite[(10.35)]{KMRT}, the formula for $e_2(h)$ in \eqref{4.14.2} shows that
\[
e_2(h)=\big(\cD(h_K)\big)=\big(\cD(h)_K\big)\in \rH^2(K).
\]In view of \eqref{4.14.1}, we can even assert that
\[
e_2(h)=\big(\cD(h)\big)\in \rH^2(F)=\Br(F)[2].
\]That is, the invariant $e_2(h)$ (when it is defined) coincides with the Brauer class of the discriminant algebra $\cD(h)$.  \hfill $\blacksquare$

\begin{thm}\label{4.15}
Let $(D,\theta)$ be a CSIA of unitary type over a field $F$ with center $F'=Z(D)$. Let $h$ be a nondegenerate hermitian form over $(D,\theta)$.

If $\car(F)\neq 2$, then $h$ is hyperbolic if and only if $e_n(h)=0$ for every $n\geqslant 1$.
\end{thm}
\begin{proof}
If $F'$ is not a field, then $h$ is necessarily hyperbolic and $e_n(h)=0$ for all $n\geqslant 1$. So it suffices to treat the case where $F'$ is a field.

Let $K$ be the field chosen as in \eqref{4.12}. By the unitary case of Theorem\;\ref{2.15}, $h$ is hyperbolic if $h_K$ is hyperbolic. So as in the proof of Theorem\;\ref{4.5}, the result follows from the case of quadratic forms.
\end{proof}

The following analogue of Prop.\;\ref{4.6} can be proved in the same way.

\begin{prop}\label{4.16}
  Let $(D,\theta)$ be a CSIA of unitary type over a field $F$ and let $E=\End_D(V_0)$ for some finitely generated $D$-module $V_0$ of constant semisimple rank. Let $h_0$ be a nondegenerate  $\theta$-hermitian form on $V_0$ and put $\tau=\ad_{h_0/\theta}$.

  Then for any nondegenerate even $\tau$-hermitian form $g$ over $(E,\tau)$, letting $h=h_0\star g$ be the hermitian form over $(D,\theta)$ obtained from $g$ by applying the Morita functor, we have $e_n(g)=e_n(h)$ for all $n\geqslant 1$.
\end{prop}

\newpara\label{4.17} Let $(D,\theta)$ and $K$ be as in \eqref{4.12}.  Recall from \eqref{2.3} that $\rW(D,\theta)$ denotes the Witt group of hermitian forms over $(D,\theta)$. Since the form $f$ to realize $\theta_K$ as an adjoint involution is fixed for all hermitian forms over $(D,\theta)$, it is clear from the construction that the $e_1$-invariant defines a group homomorphism
\[
e_1: \rW(D,\theta)\longrightarrow \rH^1(D/F).
\]We define
\[
\rI^1(D,\theta):=\rW(D,\theta),
\]and
\[
\rI^n(D,\theta):=\{[h]\in \rW(D,\theta)\,|\, e_1(h)=\cdots=e_{n-1}(h)=0\}
\] for each $n\geqslant 2$. Then we have functorial group homomorphisms
\[
e_n:\; \rI^n(D,\theta)\lra \rH^n(D/F).
\]These homomorphisms are compatible with Morita equivalence, by Prop.\;\ref{4.16}. \hfill $\blacksquare$

\medskip

\newpara\label{4.18} Let $F'$ be a quadratic separable algebra over a field $F$ and let $B$ be a separable $F$-algebra with center $Z(B)=F'$. If $F'$ is a field, then $B$ is a central simple algebra over $F'$, so the degree $\deg(B)$, the index $\ind(B)$ and the coindex $\coind(B)$ can be defined in the usual way for central simple $F'$-algebras. If $F'\cong F\times F$, then $B\cong B_0\times B_0^{op}$ for some central simple $F$-algebra $B_0$. In this case, we define the degree, index and coindex of $B$ by
\[
\deg(B)=\deg(B_0)\,,\;\ind(B)=\ind(B_0)\quad\text{and}\quad \coind(B)=\coind(B_0).
\]
We can choose a separable $F$-algebra $D$ with center $F'$ and a finitely generated right $D$-module $V$ of constant semisimple rank  such that $B=\End_D(V)$. We assume that there exist unitary $F'/F$-involutions on $D$ (hence also on $B$), and we fix such an involution $\theta$.

For any $F'/F$-involution $\sigma$ on $B$, there is a nondegenerate $\theta$-hermitian form $h$ on $V$ such that $\sigma=\ad_{h/\theta}$. This form $h$, with $\rk(h)=\coind(B)$, is uniquely determined up to a scalar multiple from $F$. So we can define the $e_n$-\emph{invariant} of $\sigma$ for each $n\geqslant 1$ by
\[
e_n(\sigma):=e_n(h)\in \rH^n(D/F)=\rH^n(D/F).
\]Here, as before, when $n\geqslant 2$ the invariant $e_n(\sigma)$ is defined only when $e_{n-1}(\sigma)=0$. By Prop.\;\ref{4.16}, this definition is independent of the choice of $(D,\theta)$.

By Thm.\;\ref{4.15}, the unitary involution $\sigma$ is hyperbolic if and only if $e_n(\sigma)=0$ for all $n\geqslant 1$.

The case with $F'\cong F\times F$ being trivial, we now assume that $F'$ is a field. By \eqref{4.14}, the following statements hold:

\begin{enumerate}
   \item The invariant $e_1(\sigma)$ descends uniquely to the base field $F$, and it is  trivial if and only if $\deg(B)$ is even.

   If $\deg(B)$ is odd, then $e_1(\sigma)\in \rH^1(F)$ corresponds to the quadratic separable extension $F'/F$.
  \item When $\deg(B)$ is even, the invariant $e_2(\sigma)$ descends uniquely to $F$, and it is equal to the Brauer class of the discriminant algebra of the involution $\sigma$. \hfill $\blacksquare$
\end{enumerate}

\newpara\label{4.19} With notation as in \eqref{4.18}, let $\sigma$ be a unitary $F'/F$-involution on $B$ and let $\Delta_{\sigma}\in\Br(F)$ be the Brauer class given by
\[
\Delta_\sigma=\begin{cases}
  0\quad & \text{ if } \deg(B) \text{ is odd},\\
  (\cD(\sigma))\quad & \text{ if } \deg(B) \text{ is even}.
\end{cases}
\](Recall that $\cD(\sigma)$ denotes the discriminant algebra of $\sigma$ when $\deg(B)$ is even.) Define
\[
N^3_{B,\sigma}(F):=\frac{\rH^3\big(F,\Q/\Z(2)\big)}{F^*\cup\Delta_\sigma+\Cor_{F'/F}\big((F')^*\cup(B)\big)}.
\]

Let $\tau$ be another $F'/F$-involution on $B$ and assume further that $\cD(\sigma)\cong\cD(\tau)$ (or equivalently $\Delta_\sigma=\Delta_\tau$) if $\deg(B)$ is even. In \cite{BarryMasqueleinQueguinerMathieu22}, a functorial relative degree 3 invariant
\[
\sR_{\sigma}(\tau)\in N^3_{B,\sigma}(F)=N^3_{B,\tau}(F)
\]is defined. We use the terminology ``\emph{Rost invariant} of $\tau$ relative to $\sigma$'' for this invariant, which is denoted by $e_3^{\sigma}(\tau)$ and called the \emph{Arason invariant} of $\tau$ relative to $\sigma$ in \cite{BarryMasqueleinQueguinerMathieu22}.  (The paper \cite{BarryMasqueleinQueguinerMathieu22} assumes the base field to have characteristic different from 2. As the reader can easily check, all the definitions and results we cite from that paper are valid in arbitrary characteristic.)

To briefly recall the definition of $\sR_{\sigma}(\tau)$, let $\bfSU(B,\sigma)$ and $\bfPGU(B,\sigma)$ denote the special unitary group and the projective unitary group of $(B,\sigma)$ respectively (cf. \cite[p.346]{KMRT}). The triple $(B,\tau,\mathrm{Id}_{F'})$ represents an element $\eta$ in the cohomology set $\rH^1\big(F,\bfPGU(B,\sigma)\big)$, and our assumption implies that $\eta$ lifts to some $\xi\in \rH^1(F,\bfSU(B,\sigma))$. Let
\[
\sR_{\bfSU(B,\sigma)}:\;\rH^1(F,\bfSU(B,\sigma))\lra \rH^3\big(F,\Q/\Z(2)\big)
\]be the Rost invariant map of the semisimple simply connected group $\bfSU(B,\sigma)$. The canonical image of $\sR_{\bfSU(B,\sigma)}(\xi)$ in the quotient group $N^3_{B,\sigma}(F)$ is independent of the choice of  $\xi$ and depends only on $\tau$. So the relative Rost invariant
\[
\sR_{\sigma}(\tau):=[\sR_{\bfSU(B,\sigma)}(\xi)]\in N^3_{B,\sigma}(F)
\]is well defined.

In the case $F'\cong F\times F$, we have $(B,\sigma)\cong (B_0\times B_0^{op},\iota)$ for some central simple $F$-algebra $B_0$ with the exchange involution $\iota$, and if moreover $\deg(B)=\deg(B_0)=2m$ is even, we have  $\Delta_\sigma=m(B_0)$ by \cite[p.129, (10.31)]{KMRT}. So we have
\[
  N^3_{B,\sigma}(F)=\frac{\rH^3\big(F,\Q/\Z(2)\big)}{F^*\cup(B_0)}
\]in this case. Moreover, the natural map $\rH^1(F,\bfSU(B,\sigma))\to \rH^1(F,\bfPGU(B,\sigma))$ is trivial since it factors through
\[
\rH^1(F,\bfU(B,\sigma))=\rH^1(F,\bfU\big(B_0\times B_0^{op},\iota)\big)=\rH^1(F,\bfGL_1(B_0))=1.
\] Hence
\begin{equation}\label{4.19.1}
   \sR_{\sigma}(\tau)=0\in N^3_{B,\sigma}(F)=\frac{\rH^3\big(F,\Q/\Z(2)\big)}{F^*\cup(B_0)}
\end{equation}
when $F'\cong F\times F$.

Let $g$ and $h$ be nondegenerate $\theta$-hermitian forms on  $V$ such that $\sigma=\ad_{g/\theta}$ and $\tau=\ad_{h/\theta}$. Still under the assumption that  $\cD(\sigma)\cong\cD(\tau)$ when $\deg(B)$ is even, we can define the \emph{relative Rost invariant} $\sR_g(h)$ of $h$ relative to $g$ by
\[
\sR_g(h):=\sR_{\sigma}(\tau)\in N^3_{B,\sigma}(F)=N^3_{B,\tau}(F).
\]As we have seen in \eqref{4.19.1}, $\sR_g(h)=0\in N^3_{B,\sigma}(F)$ if $F'$ is not a field. Note also that if $\Delta_{\sigma}=0$ (i.e. $\deg(B)$ is odd or $\deg(B)$ is even with $\cD(\sigma)$ split), then this relative Rost invariant $\sR_g(h)$ lies in
\[
N^3_{B,\sigma}(F)=\frac{\rH^3\big(F,\Q/\Z(2)\big)}{\Cor_{F'/F}((F')^*)\cup(B)}.
\]In this case, the above definition of $\sR_g(h)$ coincides with the one used in \cite[Appendix]{PaPr} and \cite[\S\,2.2]{Preeti13}.

Let $K=F(X)$ be the function field of $X=\mathrm{R}_{F'/F}\SB(D)$. By \eqref{4.14.2}, we have
$e_1(-g\perp h)=0$, so that the invariant
\[
e_2(-g\perp h)\in \rH^2\big(F,\Q/\Z(2)\big)=\rH^2_{nr}\big(X,\Q/\Z(2)\big)
\] is defined.

If $\deg(B)$ is even, then $e_1(g)=e_1(h)=0$, and the assumption $\cD(\sigma)=\cD(\tau)$ means that
\[
e_2(-g\perp h)=e_2(h)-e_2(g)=0 \in \rH^2\big(F,\Q/\Z(2)\big)=\rH^2_{nr}\big(X,\Q/\Z(2)\big).
\]Thus, we can define a \emph{relative $e_3$-invariant} $e_3(h/g)$ by
\begin{equation}\label{4.19.2}
e_3(h/g):=e_3(-g\perp h)\in \rH^3_{nr}\big(X,\Q/\Z(2)\big)\subseteq \rH^3\big(K,\Q/\Z(2)\big).
\end{equation}If moreover $\cD(\sigma)$ is split (i.e. $\Delta_\sigma=0$), then $e_3(g)$ and $e_3(h)$ are also defined, and we have
\[
e_3(h/g)=e_3(-g\perp h)=e_3(h)-e_3(g)\in \rH^3\big(K,\Q/\Z(2)\big).
\]

Let $\bar{e}_3(h/g)$ be the canonical image of $e_3(h/g)$ under the quotient map
\[
\rH^3\big(K,\Q/\Z(2)\big)\lra N^3_{B,\sigma}(K)=\frac{\rH^3\big(K,\Q/\Z(2)\big)}{K^*\cup(\Delta_\sigma)_K}.
\]Then, by \cite[Lemma\;2.6]{BarryMasqueleinQueguinerMathieu22},
\begin{equation}\label{4.19.3}
  \bar{e}_3(h/g)=\Res_{K/F}(\sR_g(h))\in N^3_{B,\sigma}(K).
\end{equation}

If $\deg(B)$ is odd, then
\[
N^3_{B,\sigma}(K)=\rH^3\big(K,\Q/\Z(2)\big)
\]because $\Delta_\sigma=0$ and $(B_K)=0$. Although the invariant $e_3(-g\perp h)\in \rH^3\big(K,\Q/\Z(2)\big)$ may not be well defined, we can use \eqref{4.19.3} to define the relative invariant
\begin{equation}\label{4.19.4}
  e_3(h/g):=\Res_{K/F}(\sR_g(h))\in N^3_{B,\sigma}(K)=\rH^3\big(K,\Q/\Z(2)\big).
\end{equation}Since $\sR_g(h)$ descends to $F$, the invariant $e_3(h/g)$ in \eqref{4.19.4} still lies in the subgroup  $ \rH^3_{nr}\big(X,\Q/\Z(2)\big)$, just as in the case of \eqref{4.19.2}. \hfill $\blacksquare$

\begin{remark}\label{4.20}
  In the context of \eqref{4.19}, if $\coind(B)=\rk(h)$ is even, then by taking $g$ to be a hyperbolic form in \eqref{4.19.3} and \eqref{4.19.4} we see that
  the invariant $e_3(h)\in \rH^3_{nr}(X,\Q/\Z(2))$ (when it is defined) descends to $F$. If moreover $\ind(B)\leqslant 2$, we can view $e_3(h)$ as an element of the group
  \[
  \frac{\rH^3\big(F,\Q/\Z(2)\big)}{\Cor_{F'/F}((F')^*\cup (D))},
  \] by Thm.\;\ref{3.5} (4). \hfill $\blacksquare$
\end{remark}

\section{Cohomological invariants of hermitian pairs}

In this section, let $F$ be a field of arbitrary characteristic and let $(D,\theta)$ be a CSIA over $F$. Assume that $\theta$ is  orthogonal (resp. symplectic) if $\car(F)\neq 2$ (resp. $\car(F)=2$).

It may not be necessary to fix the type of $(D,\theta)$ as above. But just for convenience, our choice ensures that $\theta$ and the adjoint involution $\ad_{h/\theta}$ are of the same type for any even $\theta$-hermitian form $h$. In characteristic $\neq 2$, if $\theta$ is chosen symplectic, hermitian forms should be replaced by skew-hermitian forms. In characteristic 2, the requirement that $h$ be even guarantees that $\ad_{h/\theta}$ is symplectic as long as $\theta$ is of the first kind.

\subsection{Construction of invariants}

\newpara\label{5.1} Let $(W,b)$ be a nondegenerate symmetric bilinear form over a field $K$. Assume that $b$ is alternating (so that $\dim W$ is even) if $\car(K)=2$. Let $f:\Sym(\End_K(W),\ad_{b/\mathrm{id}_K})$ be a semi-trace (cf. \eqref{2.6}) for the involutorial algebra $(\End_K(W),\ad_{b/\mathrm{id}_K})$. The pair $(b,f)$ is a hermitian pair over $(K,\mathrm{id}_K)$. According to the proof of \cite[(5.11)]{KMRT}, by using the standard identification
\[
\varphi_b:W\otimes_KW\overset{\sim}{\longrightarrow}\End_K(W);\quad \varphi_b(v\otimes w)(x):=vb(w,x),
\]we can associate to $(b,f)$ the quadratic form $q:W\to F$ given by $q(v)=f\circ\varphi_b(v\otimes v)$. Moreover, the central simple $K$-algebra with quadratic pair $(\mathrm{End}_K(W),\ad_{b/\mathrm{id}_K},f)$ is isotropic (resp. hyperbolic)  if and only if the quadratic form $q$ is isotropic (resp. hyperbolic), by \cite[(6.6) and (6.13)]{KMRT}. \hfill $\blacksquare$

\medskip

\newpara\label{5.2} Let $K$ be a nice extension of $F$ that is a generic index reduction field of orthogonal type for $(D)$, e.g., $K=F_+(D)$ as in Definition\;\ref{2.14}. We fix a nondegenerate symmetric bilinear form $b_0$ over $K$ such that $\theta_K=\ad_{b_0/\mathrm{id}_K}$.

Let $\fp=(h,f)$ be a hermitian pair over $(D,\theta)$ such that $\rk(\fp)\ind(D)$ is even. As was discussed in \eqref{2.11}, using the Morita functor associated to $b_0$, we can associate to $\fp_K$ a hermitian pair $\fp_K^{\fm}$ over $(K,\mathrm{id}_K)$, which determines a nondegenerate quadratic form $q_{\fp}^{\fm}$ over $K$ in the way described in \eqref{5.1}.  This quadratic form is uniquely determined by $\fp$ once $b_0$ is fixed (otherwise its similarity class is uniquely determined by $\theta$), and its dimension is given by
\[
\dim(q_{\fp}^{\fm})=\deg\big(\End_D(V)\big)=\rk(h)\ind(D)=\rk(\fp)\ind(D).
\]For every integer $n\geqslant 1$ we can define the $e_n$-\emph{invariant} of $\fp$ by
\[
e_n(\fp):=e_n(q_{\fp}^{\fm})\in \rH^n(K).
\]Here as in \eqref{4.2} and \eqref{4.12}, for $n\geqslant 2$ the invariant $e_n(\fp)$ is defined only when $e_{n-1}(\fp)=0$, and these $e_n$-invariants are functorial in the base field. \hfill $\blacksquare$

\begin{lemma}\label{5.3}
  With notation and hypotheses as in $\eqref{5.2}$, we have
  \[
e_n(\fp)\in \rH^n_{+}(D)=\rH^n_{nr}\big(K/F,\,\Z/2(n-1)\big) \quad \text{ for every }\; n\geqslant 1.
  \]
\end{lemma}
\begin{proof}This is proved in the same way as in the proof of Lemma\;\ref{4.3}.
\end{proof}

\newpara\label{5.4} Retain the notation and hypotheses of \eqref{5.2}. By Thm.\;\ref{3.4} (1) and (2), we have canonical isomorphisms
\[
 \rH^1(F)\cong \rH^1_+(D)\quad \text{and}\quad \frac{\rH^2(F,\,\Z/4(1))}{\langle (D)\rangle}\cong\rH^2_+(D).
\](Note that the Brauer class $(D)\in\Br(F)$ is 2-torsion.) Therefore, $e_1(\fp)$ descends uniquely to the base field  $F$ and $e_2(\fp)$ descends to a 4-torsion Brauer class over $F$.

If $\ind(D)\leqslant2$, then by Thm.\;\ref{3.4} (4),
\[
\frac{\rH^n(F,\,\Z/4(n-1))}{\rH^{n-2}(F,\Z/2(n-2))\cup (D)}\cong\rH^n_+(D)
\]
for every $n\geqslant 3$, so in this case $e_n(\fp)$ (when it is defined) descends to a 4-torsion class over $F$. \hfill $\blacksquare$

\medskip

It is clear that our construction of the $e_n$-invariants generalizes that of Berhuy \cite{Berhuy07ArchMath} to arbitrary characteristic and arbitrary index. We also have the following generalization of \cite[Thm.\;13]{Berhuy07ArchMath}.

\begin{thm}\label{5.5}
Let $\fp$ be a hermitian pair over $(D,\theta)$ such that $\rk(\fp)\ind(D)$ is even.
Then $\fp$ is hyperbolic if and only if $e_n(\fp)=0$ for all $n \geqslant 1$.
\end{thm}
\begin{proof}
For any generic index reduction field $K$  of orthogonal type of $D$, the hermitian pair $\fp_K$ is hyperbolic if and only if the associated quadratic form $q_{\fp}^{\fm}$ is hyperbolic. So it suffices to apply Theorems\;\ref{2.15} and \ref{2.16}.
\end{proof}

\newpara\label{5.6} Let  $(E,\tau)=(\End_D(V_0),\,\ad_{h_0/\theta})$, where $(V_0,h_0)$ is a nondegenerate even hermitian form over $(D,\theta)$. For a hermitian pair $\mathfrak{P}$ over $(E,\tau)$, if $\fp$ denotes the hermitian pair over $(D,\theta)$ obtained by applying the Morita functor to $\mathfrak{P}$, then,  similar to the symplectic and unitary cases as shown in Props.\;\ref{4.6} and \ref{4.16},   $e_n(\mathfrak{P})=e_n(\fp)$ for every $n\geqslant 1$.

If $(A,\sigma,f)$ is a CSAQP such that $(A)=(D)\in \Br(F)$ and $\deg(A)$ is even, then we can define the $e_n$-\emph{invariant} of $(A,\sigma,f)$ by
\[
e_n(A,\sigma,f):=e_n(h,f)\,,\quad n\geqslant 1\,,
\]where $(h,f)$ is a hermitian lift of $(A,\sigma,f)$ over $(D,\theta)$. By the Morita invariance we have just seen above, this definition is independent of the choice of $(D,\theta)$.

Let $\rW_{\hp}(D,\theta)$ be the Witt group of hermitian pairs over $(D,\theta)$, as defined in \eqref{2.11}. Define
\[
\rI^1_{\hp}(D,\theta):=\{\fp\in\rW_{\hp}(D,\theta)\,|\,\rk(\fp)\ind(D) \text{ is even}\}
\]and
\[
\rI^n_{\hp}(D,\theta):=\{\fp\in\rW_{\hp}(D,\theta)\,|\,\rk(\fp)\ind(D) \text{ is even and } e_1(\fp)=\cdots=e_{n-1}(\fp)=0\}
\]
for each $n\geqslant 2$. Then, similar to the symplectic and unitary cases discussed in \eqref{4.7} and \eqref{4.17} respectively, our $e_n$-invariants induce group homomorphisms
\[
e_n:\rI^n_{\hp}(D,\theta)\longrightarrow \rH^n_+(D)\,,\quad n\geqslant 1,
\]which are compatible with Morita equivalence. \hfill $\blacksquare$

\subsection{Link with classical invariants}\label{sec5p2}

\newpara\label{5.7} For a central simple algebra with quadratic pair  $(A,\sigma,f)$ over $F$, let its \emph{discriminant} $\mathrm{disc}(A,\sigma,f) \in \rH^1(F,\Z/2)$ be defined as in  \cite[\S\,7.B]{KMRT}. The \emph{Clifford algebra} $C(A,\sigma,f)$ of $(A,\sigma,f)$ is defined in \cite[(8.7)]{KMRT}. If $\deg(A)$ is even and $\mathrm{disc}(A,\sigma,f)=0$, the Clifford algebra $C(A,\sigma,f)$ decomposes as a direct product $C^+(A,\sigma,f) \times C^-(A,\sigma,f)$ of two central simple $F$-algebras (\cite[(8.10)]{KMRT}). In this case, we define the \emph{Clifford invariant} of $(A,\sigma,f)$ by
\[
\mathrm{Cl}(A,\sigma,f) = [C^+(A,\sigma,f)] = [C^-(A,\sigma,f) ] \in \frac{\mathrm{Br}(F)}{\langle(A)\rangle} = \frac{\rH^2(F,\mathbb{Q}/\mathbb{Z}(1))}{\langle(A)\rangle}.
\]

For a hermitian pair $\fp=((V,h),f)$ over $(D,\theta)$ with $\rk(\fp)\ind(D)$ even, we define the \emph{discriminant} $\mathrm{disc}(\fp)$ of $\fp$ to be the discriminant of its associated CSAQP, namely,
 \[
\disc(h,f):=\disc(\mathrm{End}_D(V),\mathrm{ad}_h,f).
 \] Since in the split case the discriminant of a CSAQP coincides with that of the associated quadratic form (\cite[(7.8)]{KMRT}),  we see that
\[
 e_1(\fp) = \mathrm{disc}(\fp)\in \rH^1(F)=\rH^1_+(D).
\] If $\disc(\fp)=0$, we can further define the \emph{Clifford invariant} $\mathrm{Cl}(\fp)$ to be the Clifford invariant of the CSAQP associated to $\fp$, i.e.,
\[
\mathrm{Cl}(h,f) := \mathrm{Cl}(\mathrm{End}_D(V),\mathrm{ad}_h,f) .
\]Because of the compatibility between the Clifford algebra of a quadratic form and its associated CSAQP (\cite[(8.8) and (8.11)]{KMRT}), we have
\[
e_2(\fp) = \mathrm{Cl}(\fp)\in \frac{\rH^2(F,\mathbb{Z}/4(1))}{\langle(D)\rangle}=\rH^2_+(D).
\]when $e_1(\fp)=\disc(\fp)$ vanishes. \hfill $\blacksquare$

\medskip

\newpara\label{5.8} Let $(h,f)$ be a hermitian pair of rank $m$ over $(D,\theta)$. Define the algebraic groups (cf. \cite[\S\,23.B]{KMRT})
\[
\begin{split}
 \bfU (h, f ) & := \bfO(\mathrm{End}_D(V),\mathrm{ad}_h,f), \\
 \bfSU(h,f) & := \bfO^+(\mathrm{End}_D(V),\mathrm{ad}_h,f),\\
 \bfSpin (h, f ) & := \bfSpin(\mathrm{End}_D(V),\mathrm{ad}_h,f) .
\end{split}
\]
There are short exact sequences \begin{equation}
\label{5.8.1}
1 \longrightarrow \bfSU(h,f) \longrightarrow\bfU(h,f) \longrightarrow\mathbb{Z}/2 \longrightarrow 1
\end{equation} and \begin{equation}
\label{5.8.2}
1 \longrightarrow \mu_2 \longrightarrow \bfSpin(h,f) \longrightarrow \bfSU(h,f) \longrightarrow 1.
\end{equation}

 The sequence \eqref{5.8.1} induces a map
 \[ \mathrm{disc}_{(h,f)} : \rH^1(F,\bfU (h,f))\longrightarrow\rH^1(F,\mathbb{Z}/2) \]
 called the \emph{relative discriminant with respect to $(h,f)$}. Note that the set $\rH^1(F,\bfU (h,f))$ can be identified with the set of isomorphism classes of hermitian pairs of the same rank $m$ over $(D,\theta)$. So if $((V',h'),f')$ is a hermitian pair of rank $m$, then it represents a class $[h',f']\in \rH^1(F,\bfU(h,f))$, and as in \cite{BayerFluckigerParimala} and \cite[\S\,7.4.1, p.112]{Gille}, we may define the relative discriminant
  \[
  \disc_{(h,f)}(h',f'):=\disc_{(h,f)}([h',f']).
  \]
 If there is an orthogonal decomposition $(h,f)=(h_1,f_1)\perp (h_2,f_2)$, and if $(h'_i,f_i')$ is a hermitian pair of the same rank as $(h_i,f_i)$, then we have (cf. \cite[\S\,7.4.1, p.112]{Gille})
\begin{equation}\label{5.8.3}
  \disc_{(h,f)}((h_1',f'_1)\perp (h_2', f'_2))=\disc_{(h_1,f_1)}(h_1',f_1')+\disc_{(h_2,f_2)}(h_2',f'_2).
\end{equation}

Let $\mathbb{H}_{2m}$ denote the hyperbolic hermitian pair over $(D,\theta)$ of rank $2m$. Then
\[
(h,f)\perp (-h,f)\cong \mathbb{H}_{2m}\cong (h',f')\perp (-h',f').
\]This together with \eqref{5.8.3} shows
\[
\begin{split}
\disc_{(h,f)}(h',f')&=\disc_{(h,f)}(h',f')+\disc_{(-h,f)}(-h,f)\\
&=\disc_{\mathbb{H}_{2m}}((h',f')\perp (-h,f))=\disc((h',f')\perp (-h,f)).
\end{split}
\]Here we have used the fact that
\begin{equation}\label{5.8.4}
\mathrm{disc}(h_0,f_0) = \mathrm{disc}_{\mathbb{H}_{2m}}(h_0,f_0)
\end{equation}for any hermitian pair $(h_0,f_0)$ of rank $2m$. To prove \eqref{5.8.4}, by base extension to the generic splitting field $F_+(D)$, we may reduce to the case of quadratic forms, in which case the result is classical (cf. \cite[\S\,IV.5, p.227]{Knus}).

If moreover $\ind(D)m=2n$ is even, then $\disc(h',f')$ and $\disc(-h,f)$ are defined and we have
\[
\disc((h',f')\perp (-h,f))=\disc(h',f')+\disc(-h,f).
\]It is clear from the definition that $\disc(-h,f)=\disc(h,f)$. So in this case we get
\begin{equation}\label{5.8.5}
\disc_{(h,f)}(h',f')=\disc(h',f')+\disc(h,f).
\end{equation}


We can also define a relative Clifford invariant for hermitian pairs with trivial relative discriminant. In fact, the sequence \eqref{5.8.2} induces a map
\[
\mathrm{Cl}_{(h,f)} : \rH^1\big(F,\bfSU (h,f)\big)\Big/ (\mathbb{Z}/2) \to \rH^2_{\mathrm{fppf}}(F,\mu_2)\Big/\langle (D)\rangle = \rH^2(F,\mathbb{Z}/2(1))\Big/\langle (D)\rangle\,,
\] called the \emph{relative Clifford invariant with respect to $(h,f)$}. Here $\rH^1(F,\bfSU (h,f))/ (\mathbb{Z}/2)$ means the quotient of $\rH^1(F,\bfSU (h,f))$ by the action of $\mathbb{Z}/2$ via the map $\mathbb{Z}/2 \to \rH^1(F,\bfSU (h,f))$ induced by \eqref{5.8.1}. If $(h',f')$ is a hermitian pair with $\disc_{(h,f)}(h',f')=0$, then it represents a class $[h',f']\in \rH^1(F,\bfSU (h,f))/ (\mathbb{Z}/2)$, so we can define
\[
\mathrm{Cl}_{(h,f)}(h',f'):=\mathrm{Cl}_{(h,f)}([h',f']).
\]
As an analog of \eqref{5.8.4}, we have
\begin{equation}\label{5.8.6}
\mathrm{Cl}(h_0,f_0) = \mathrm{Cl}_{\mathbb{H}_{2m}}(h_0,f_0)
\end{equation}for any hermitian pair $(h_0,f_0)$ of rank $2m$ with $\disc(h_0,f_0)=0$. If  $\ind(D)m=2n$ is even and $\disc(h,f)=\disc(h',f')=0$, we have
\begin{equation}\label{5.8.7}
  \mathrm{Cl}_{(h,f)}(h',f')=\mathrm{Cl}(h',f')+\mathrm{Cl}(h,f),
\end{equation}by \cite[Lem.\;7.4.6]{Gille}.\hfill $\blacksquare$

\begin{lemma}\label{5.9}
Let $(h,f)$ and $(h',f')$ be hermitian pairs of the same rank over $(D,\theta)$, and suppose
\[ \mathrm{disc}_{(h,f)}(h',f') = 0 \text{ and } \mathrm{Cl}_{(h,f)}(h',f') =0 . \]

\begin{enumerate}
\item The class $[(h',f')] \in \rH^1(F, \bfU (h,f))$ lifts to an element $\xi \in \rH^1(F,\bfSpin (h,f))$.
\item Any two lifts $\xi$, $\xi'$ of $[(h',f')]$ have the same image in $\rH^1(F,\bfSU (h,f))$.
\end{enumerate}
\end{lemma}
\begin{proof}
$(1)$ The cohomological exact sequence induced by \eqref{5.8.1} shows that $[h',f'] \in \rH^1(F,\bfU(h,f))$ comes from an element $\eta \in \rH^1(F,\bfSU(h,f))$. The assumption on $\mathrm{Cl}_{(h,f)}(h',f')$ means that the connecting map
\[
\delta : \rH^1(F,\bfSU(h,f)) \to \rH^2_{\mathrm{fppf}}(F,\mu_2)=\Br(F)[2]
\] arising from the cohomology of \eqref{5.8.2} sends $\eta$ to $0$ or $(D)$. If $\delta(\eta)=0$, the statement is clear. So assume $\delta(\eta) \neq 0$. Let $\eta' \in \rH^1(F,\bfSU(h,f))$ be the element obtained from $\eta$ by the action of $1\in \mathbb{Z}/2$. Then $\eta'$ is another lift of $[h',f']$. Thanks to \cite[Lem.\;7.4.5 (2)]{Gille}, $\delta(\eta')=\delta(\eta) + (D) =0 \in \rH^2_{\mathrm{fppf}}(F,\mu_2)$. Thus $\eta'$ comes from an element $\xi \in \rH^1(F,\bfSpin (h,f))$.

$(2)$ Let $\rho : \bfSpin (h,f) \to \bfSU(h,f)$ be the canonical homomorphism. If $\eta := \rho_\ast(\xi)$ and $\eta' := \rho_\ast(\xi')$ are distinct, then they are the image of each other by the action of $1\in \mathbb{Z}/2$, because they both lift $[(h',f')]$. By
\cite[Lem.\;7.4.5 (2)]{Gille}, it follows that $\delta(\eta)-\delta(\eta') \neq 0 \in \rH^2_{\mathrm{fppf}}(F,\mu_2)$. But
 \[ \delta(\eta) = \delta(\rho_\ast(\xi)) =0 = \delta(\rho_\ast(\xi')) =  \delta(\eta'), \]
 hence a contradiction.
\end{proof}

\newpara\label{5.10} Let $(h,f)$ and $(h',f')$ be hermitian pairs of the same rank $m$ over $(D,\theta)$, and assume that $\ind(D)m=2n\geqslant 4$ is even. In this case the group $\bfSpin(h,f)$ is a semisimple simply connected group of type $D_n$, so it has a Rost invariant map
 \[ \sR_{(h,f)}:=\sR_{\bfSpin(h,f)} : \rH^1\big(F,\bfSpin(h,f)\big) \longrightarrow \rH^3\big(F,\mathbb{Q}/\mathbb{Z}(2)\big) .\]
Let $i_*: \rH^1(F,\mu_2) \to \rH^1\big(F,\bfSpin(h,f)\big)$ be the natural map induced by the inclusion $\mu_2 \to \bfSpin(h,f)$. Then for every $x \in \rH^1(F,\mu_2)$ we have
\begin{equation}\label{5.10.1}
  R_{(h,f)} \circ i_\ast (x) = x \cup (D),
\end{equation}by \cite[p.441]{KMRT} and \cite[Example\;10.4]{GaribaldiMerkurjev17}. (The formula in \cite[(10.5)]{GaribaldiMerkurjev17} should be corrected for $n$ odd, noticing that $\mu_2$ is not the center of the spin group.)

Now suppose moreover that $\mathrm{disc}_{(h,f)}(h',f')=0$ and $\mathrm{Cl}_{(h,f)}(h',f')=0$.

Then, by Lemma \ref{5.9} (1), the class $[(h',f')] \in \rH^1(F, \bfU (h,f))$ is the image of some $\xi \in \rH^1(F,\bfSpin (h,f))$ under the composite map
\[
\rH^1\big(F,\bfSpin(h,f)\big) \longrightarrow \rH^1\big(F,\bfSU(h,f)\big) \longrightarrow \rH^1\big(F,\bfU(h,f)\big) . \]
Combining Lemma\;\ref{5.9} (2) with the formula \eqref{5.10.1}, we can deduce from a well known twisting formula \cite[Lem.\;7]{Gille00Ktheory} that the \emph{relative Rost invariant}
\[
\sR_{(h,f)}(h',f') := \sR_{(h,f)}(\xi) \in \frac{\rH^3\big(F,\mathbb{Q}/\mathbb{Z}(2)\big)}{F^* \cup (D)}
\]is well defined. This invariant is Morita invariant, because the group $\bfSpin(h,f)$ depends only on the CSAQP associated to $(h,f)$.

If $(h_0,f_0)$ is a hermitian pair of even rank $2m$ with trivial discriminant and trivial Clifford invariant, we can define its \emph{Rost invariant} by
\[
\sR(h_0,f_0):=\sR_{\mathbb{H}_{2m}}(h_0,f_0)\in \frac{\rH^3\big(F,\mathbb{Q}/\mathbb{Z}(2)\big)}{F^* \cup (D)} .
\]

Suppose that $e_1(h,f)=\disc(h,f)=0$ and $e_2(h,f)=\mathrm{Cl}(h,f)=0$, so that the invariant $e_3(h,f)\in \rH^3(K)$ is defined over $K:=F_+(D)$. Note that
\[
\rk(h_K,f_K)=\ind(D_K)\rk(h_K,f_K)=\ind(D)m=2n.
\]So the Rost invariant $\sR(h_K,f_K)=\sR_{\mathbb{H}_{2n}}(h_K,f_K)\in \rH^3(K)$ is defined. Since in the split case this Rost invariant coincides with the Arason invariant for quadratic forms (\cite[p.107, Example\;2.3 and p.144, Remark\;13.6]{Merkurjev03inGMScohinv}), we can conclude by Morita invariance that
\[
  e_3(h,f)=\sR(h_K,f_K)\in \rH^3(K).
\]
If moreover $(h,f)$ has even rank, so that $\sR(h,f)$ is defined, then the above formula shows that
\begin{equation}\label{5.10.2}
  e_3(h,f)=\Res_{K/F}\big(\sR(h,f)\big)\in \rH^3(K).
\end{equation}In particular, the invariant $e_3(h,f)$ descends to the base field $F$ in this case; we can even view $e_3(h,f)$ as an element of the quotient group $ \rH^3\big(F,\mathbb{Q}/\mathbb{Z}(2)\big)/F^* \cup (D)$ if $\ind(D)\,|\,4$ or $D$ is a decomposable division algebra, by Thm.\;\ref{3.4} (4).  \hfill $\blacksquare$

\begin{remark}\label{5.11}
As shown in  \cite[\S\,5.A]{BarryMasqueleinQueguinerMathieu22}, under an extra condition one can define a relative Rost invariant for quadratic pairs, in a way that is compatible with our definition for hermitian pairs. \hfill $\blacksquare$
\end{remark}

\subsection{A hyperbolicity criterion in the quaternion case}

For any prime number $p$, let $\sd_p(F)$ denote the \emph{separable $p$-dimension} of $F$ as defined in \cite[p.62]{Gille00Ktheory}.

\begin{prop}\label{5.12}
Assume that $\mathrm{ind}(D)=2$ and $\mathrm{sd}_2(F) \leqslant 3$. Then a hermitian pair $\fp$ over $(D,\theta)$ is hyperbolic if and only if
$e_1(\fp) = e_2(\fp) = e_3(\fp) =0$.
\end{prop}
\begin{proof}
By Thm.\;\ref{3.4}, the assumption $\ind(D)=2$ guarantees that $\rH^n_+(D)$ is a quotient of the 2-primary torsion subgroup of $\rH^n(F,\Q/\Z(n-1))$ for every $n\geqslant 1$.
But if $\mathrm{sd}_2(F) \leqslant 3$, $\rH^n(F,\Q/\Z(n-1))$ has no 2-torsion for every $n\geqslant 4$.
So the result follows from Theorem\;\ref{5.5}.
\end{proof}

Prop.\;\ref{5.12} allows us to generalize \cite[Prop.\;15]{Berhuy07ArchMath} to arbitrary characteristic as follows.

\begin{thm}\label{5.13}
Assume that $\mathrm{ind}(D) \leqslant 2$ and $\mathrm{sd}_2(F) \leqslant 3$. Then a hermitian pair of even rank over $(D,\theta)$ is hyperbolic if and only if it has
trivial discriminant, trivial Clifford invariant and trivial Rost invariant.
\end{thm}
\begin{proof}
  The case $\ind(D)=1$ reduces to the well known case for quadratic forms. We may thus assume $\ind(D)=2$. In this case, the invariants $e_1,e_2$ and $e_3$ agree with the corresponding classical invariants even over the base field, as we have seen in \eqref{5.7} and \eqref{5.10}. So the theorem follows from Prop.\;\ref{5.12}.
\end{proof}

\noindent \emph{Acknowledgements.}  We thank Philippe Gille, Ting-Yu Lee, Anne Qu\'eguiner-Mathieu and Zhengyao Wu for helpful discussions.

\addcontentsline{toc}{section}{\textbf{References}}

\bibliographystyle{alpha}
\bibliography{HuLourdeaux}

\

Yong HU

\medskip

Department of Mathematics

Southern University of Science and Technology


Shenzhen 518055, China


Email: huy@sustech.edu.cn

\

Alexandre LOURDEAUX

\medskip

SUSTech International Center for Mathematics

Southern University of Science and Technology


Shenzhen 518055, China


Email: alexandre.lourdeaux@outlook.com


\end{document}